\documentclass[12pt,letterpaper]{article}
\usepackage{amssymb, amsthm, amsmath}
\usepackage{tikz}
\usepackage{hyperref}
\usepackage[margin=1in]{geometry}

\allowdisplaybreaks

\usetikzlibrary{arrows.meta}

\def\picERa{
\begin{tikzpicture}
\begin{scope}
    \clip (0,0.25) rectangle (1.5,2.5);
    \draw[ultra thick]         (0,2)--(1.5,2);
   
    \draw[ultra thick] 
        (0,1) to[out=0,in=90] (0.75,0.25) to[out=90,in=180] (1.5,1.5);
     
    \draw[ultra thick] 
        (0,1.5) to[out=0,in=90] (0.75,0.25) to[out=90,in=180] (1.5,1);
\end{scope}
\draw[ultra thick,color=white!25!black,dashed] (0,0.25)--(1.5,0.25);
\draw[ultra thick,color=white!25!black,fill=white] (0.75,0.25) circle (3pt);
\draw[rounded corners] (0,0) rectangle (1.5,2.5);

\end{tikzpicture}}

\def\picERc{
\begin{tikzpicture}
\begin{scope}
    \clip (0,0.25) rectangle (1.5,2.5);
    \draw[ultra thick] 
        (0,2) to[out=0,in=180] (1.5,2);
    \draw[ultra thick] 
        (0,1.5) to[out=0,in=180] (1.5,1.5);   
    \draw[line width = 3pt, color=white]
        (0,1) to[out=0,in=90] (0.75,0.25) to[out=90,in=180] (1.5,1);
    \draw[ultra thick] 
        (0,1) to[out=0,in=90] (0.75,0.25) to[out=90,in=180] (1.5,1);
\end{scope}
\draw[ultra thick,color=white!25!black,dashed] (0,0.25)--(1.5,0.25);
\draw[ultra thick,color=white!25!black,fill=white] (0.75,0.25) circle (3pt);
\draw[rounded corners] (0,0) rectangle (1.5,2.5);

\end{tikzpicture}}

\def\picERd{
\begin{tikzpicture}
\begin{scope}
    \clip (0,0.25) rectangle (1.5,2.5);
    \draw[ultra thick] 
        (0,2) to[out=0,in=180] (1.5,2);
    \draw[ultra thick] 
        (0,1.5) to[out=0,in=180] (1.5,1);   
    \draw[line width = 3pt, color=white]
        (0,1) to[out=0,in=90] (0.75,0.25) to[out=90,in=180] (1.5,1.5);
    \draw[ultra thick] 
        (0,1) to[out=0,in=90] (0.75,0.25) to[out=90,in=180] (1.5,1.5);
\end{scope}
\draw[ultra thick,color=white!25!black,dashed] (0,0.25)--(1.5,0.25);
\draw[ultra thick,color=white!25!black,fill=white] (0.75,0.25) circle (3pt);
\draw[rounded corners] (0,0) rectangle (1.5,2.5);

\end{tikzpicture}}

\def\picJUGA{
\begin{tikzpicture}
\begin{scope}
    \clip (0,0.25) rectangle (1.5,2.5);
    \draw[ultra thick] 
        (0,2) to[out=0,in=180] (1.5,1.5);
    \draw[ultra thick] 
        (0,1.5) to[out=0,in=180] (1.5,1);
    \draw[line width = 5pt, color=white]
        (0,1) to[out=0,in=90] (0.75,0.25) to[out=90,in=180] (1.5,2);
    \draw[ultra thick] 
        (0,1) to[out=0,in=90] (0.75,0.25) to[out=90,in=180] (1.5,2);
    \draw[ultra thick] 
        (0,1) to[out=0,in=90] (0.75,0.25) to[out=90,in=180] (1.5,1);
\end{scope}
\draw[ultra thick,color=white!25!black,dashed] (0,0.25)--(1.5,0.25);
\draw[ultra thick,color=white!25!black,fill=white] (0.75,0.25) circle (3pt);
\draw[rounded corners] (0,0) rectangle (1.5,2.5);

\node[inner sep = 1pt, circle, draw, fill=white, ultra thick] at (0,1) {\tiny $2$};
\node[inner sep = 1pt, circle, draw, fill=white, ultra thick] at (0,1.5) {\tiny $1$};
\node[inner sep = 1pt, circle, draw, fill=white, ultra thick] at (0,2) {\tiny $1$};
\node[inner sep = 1pt, circle, draw, fill=white, ultra thick] at (1.5,1) {\tiny $2$};
\node[inner sep = 1pt, circle, draw, fill=white, ultra thick] at (1.5,1.5) {\tiny $1$};
\node[inner sep = 1pt, circle, draw, fill=white, ultra thick] at (1.5,2) {\tiny $1$};
\end{tikzpicture}}

\def\picJUGB{
\begin{tikzpicture}
\begin{scope}
    \clip (0,0.25) rectangle (1.5,2.5);
    \draw[ultra thick] 
        (0,2) to[out=0,in=180] (1.5,2);
    \draw[ultra thick] 
        (0,1.5) to[out=0,in=180] (1.5,1.5);
    \draw[line width = 5pt, color=white]
        (0,1) to[out=0,in=90] (0.75,0.25) to[out=90,in=180] (1.5,2);
    \draw[ultra thick] 
        (0,1) to[out=0,in=90] (0.75,0.25) to[out=90,in=180] (1.5,2);
    \draw[ultra thick] 
        (0,1) to[out=0,in=90] (0.75,0.25) to[out=90,in=180] (1.5,1);
\end{scope}
\draw[ultra thick,color=white!25!black,dashed] (0,0.25)--(1.5,0.25);
\draw[ultra thick,color=white!25!black,fill=white] (0.75,0.25) circle (3pt);
\draw[rounded corners] (0,0) rectangle (1.5,2.5);

\node[inner sep = 1pt, circle, draw, fill=white, ultra thick] at (0,1) {\tiny $2$};
\node[inner sep = 1pt, circle, draw, fill=white, ultra thick] at (0,1.5) {\tiny $2$};
\node[inner sep = 1pt, circle, draw, fill=white, ultra thick] at (0,2) {\tiny $1$};
\node[inner sep = 1pt, circle, draw, fill=white, ultra thick] at (1.5,1) {\tiny $1$};
\node[inner sep = 1pt, circle, draw, fill=white, ultra thick] at (1.5,1.5) {\tiny $2$};
\node[inner sep = 1pt, circle, draw, fill=white, ultra thick] at (1.5,2) {\tiny $2$};
\end{tikzpicture}}

\def\picBCCGa{
\begin{tikzpicture}
\begin{scope}
    \clip (0,0.25) rectangle (1.5,2.5);
    \draw[ultra thick] 
        (0,2) to[out=0,in=180] (1.5,2);
    \draw[line width = 3pt, color=white]
        (0,1) to[out=0,in=90] (0.5,0.25) to[out=90,in=180] (1.5,1.5);
    \draw[ultra thick] 
        (0,1) to[out=0,in=90] (0.5,0.25) to[out=90,in=180] (1.5,1.5);
    \draw[line width = 3pt, color=white]
        (0,1.5) to[out=0,in=90] (1,0.25) to[out=90,in=180] (1.5,1);
    \draw[ultra thick] 
        (0,1.5) to[out=0,in=90] (1,0.25) to[out=90,in=180] (1.5,1);
\end{scope}
\draw[ultra thick,color=white!25!black,dashed] (0,0.25)--(1.5,0.25);
\draw[ultra thick,color=white!25!black,fill=white] (0.5,0.25) circle (3pt);
\draw[ultra thick,color=white!25!black,fill=white] (1,0.25) circle (3pt);
\draw[rounded corners] (0,0) rectangle (1.5,2.5);

\end{tikzpicture}}

\def\picCb{
\begin{tikzpicture}
\begin{scope}
    \clip (0,0.25) rectangle (1.5,2.5);
    \draw[ultra thick] 
        (0,2) to[out=0,in=180] (1.5,2);
    \draw[line width = 5pt, color=white]
        (0,1) to[out=0,in=90] (0.75,0.25) to[out=90,in=180] (1.5,1);
    \draw[ultra thick] 
        (0,1) to[out=0,in=90] (0.75,0.25) to[out=90,in=180] (1.5,1);
    \draw[ultra thick]
        (0,1.5) to[out=0, in=180] (1.5,1.5);
\end{scope}
\draw[ultra thick,color=white!25!black,dashed] (0,0.25)--(1.5,0.25);
\draw[ultra thick,color=white!25!black,fill=white] (0.75,0.25) circle (3pt);
\draw[rounded corners] (0,0) rectangle (1.5,2.5);

\end{tikzpicture}}

\def\picCa{
\begin{tikzpicture}
\begin{scope}
    \clip (0,0.25) rectangle (1.5,2.5);
    \draw[ultra thick] 
        (0,2) to[out=0,in=180] (1.5,2);
    \draw[ultra thick] 
        (0,1) to[out=0,in=180] (1.5,1);
    \draw[ultra thick]
        (0,1.5) to[out=0, in=180] (1.5,1.5);
\end{scope}
\draw[ultra thick,color=white!25!black,dashed] (0,0.25)--(1.5,0.25);
\draw[ultra thick,color=white!25!black,fill=white] (0.75,0.25) circle (3pt);
\draw[rounded corners] (0,0) rectangle (1.5,2.5);

\end{tikzpicture}}

\def\picCc{
\begin{tikzpicture}
\begin{scope}
    \clip (0,0.25) rectangle (1.5,2.5);
    \draw[ultra thick] 
        (0,2) to[out=0,in=180] (1.5,2);
    \draw[ultra thick]
        (0,1.5) to[out=0, in=180] (1.5,1);
    \draw[line width = 5pt, color=white]
        (0,1) to[out=0,in=90] (0.75,0.25) to[out=90,in=180] (1.5,1.5);
    \draw[ultra thick] 
        (0,1) to[out=0,in=90] (0.75,0.25) to[out=90,in=180] (1.5,1.5);
    
\end{scope}
\draw[ultra thick,color=white!25!black,dashed] (0,0.25)--(1.5,0.25);
\draw[ultra thick,color=white!25!black,fill=white] (0.75,0.25) circle (3pt);
\draw[rounded corners] (0,0) rectangle (1.5,2.5);

\end{tikzpicture}}

\def\picCd{
\begin{tikzpicture}
\begin{scope}
    \clip (0,0.25) rectangle (1.5,2.5);
    \draw[ultra thick] 
        (0,2) to[out=0,in=180] (1.5,1.5);
    \draw[ultra thick]
        (0,1.5) to[out=0, in=180] (1.5,1);
    \draw[line width = 5pt, color=white]
        (0,1) to[out=0,in=90] (0.75,0.25) to[out=90,in=180] (1.5,2);
    \draw[ultra thick] 
        (0,1) to[out=0,in=90] (0.75,0.25) to[out=90,in=180] (1.5,2);
    
\end{scope}
\draw[ultra thick,color=white!25!black,dashed] (0,0.25)--(1.5,0.25);
\draw[ultra thick,color=white!25!black,fill=white] (0.75,0.25) circle (3pt);
\draw[rounded corners] (0,0) rectangle (1.5,2.5);

\end{tikzpicture}}

\def\picEXa{
\begin{tikzpicture}
\begin{scope}
    \clip (0,0.25) rectangle (1.5,2.5);
    \draw[ultra thick] 
        (0,1.7) to[out=0,in=180]  (1.5,1.7);
    \draw[line width = 2pt, color=white]
        (0,1.3) to[out=0,in=90] (0.75,0.25) to[out=90,in=180] (1.5,1.3);
    \draw[ultra thick] 
        (0,1.3) to[out=0,in=90] (0.75,0.25) to[out=90,in=180] (1.5,1.3);
\end{scope}
\draw[ultra thick,color=white!25!black,dashed] (0,0.25)--(1.5,0.25);
\draw[ultra thick,color=white!25!black,fill=white] (0.75,0.25) circle (3pt);
\draw[rounded corners] (0,0) rectangle (1.5,2.5);

\node[inner sep = 1pt, circle, draw, fill=white, ultra thick] at (0,1.3) {\tiny $2$};
\node[inner sep = 1pt, circle, draw, fill=white, ultra thick] at (0,1.7) {\tiny $1$};
\node[inner sep = 1pt, circle, draw, fill=white, ultra thick] at (1.5,1.3) {\tiny $2$};
\node[inner sep = 1pt, circle, draw, fill=white, ultra thick] at (1.5,1.7) {\tiny $1$};
\end{tikzpicture}}

\def\picEXb{
\begin{tikzpicture}
\begin{scope}
    \clip (0,0.25) rectangle (1.5,2.5);
    \draw[ultra thick] 
        (0,1.7) to[out=0,in=180]  (1.5,1.3);
    \draw[line width = 3pt, color=white]
        (0,1.3) to[out=0,in=90] (0.75,0.25) to[out=90,in=180] (1.5,1.3);
    \draw[line width = 3pt, color=white]
        (0.75,0.25) to[out=90,in=180] (1.5,1.7);
    \draw[ultra thick] 
        (0,1.3) to[out=0,in=90] (0.75,0.25) to[out=90,in=180] (1.5,1.3);
    \draw[ultra thick] 
    (0.75,0.25) to[out=90,in=180] (1.5,1.7);
\end{scope}
\draw[ultra thick,color=white!25!black,dashed] (0,0.25)--(1.5,0.25);
\draw[ultra thick,color=white!25!black,fill=white] (0.75,0.25) circle (3pt);
\draw[rounded corners] (0,0) rectangle (1.5,2.5);

\node[inner sep = 1pt, circle, draw, fill=white, ultra thick] at (0,1.3) {\tiny $2$};
\node[inner sep = 1pt, circle, draw, fill=white, ultra thick] at (0,1.7) {\tiny $1$};
\node[inner sep = 1pt, circle, draw, fill=white, ultra thick] at (1.5,1.3) {\tiny $2$};
\node[inner sep = 1pt, circle, draw, fill=white, ultra thick] at (1.5,1.7) {\tiny $1$};
\end{tikzpicture}}

\def\picEXc{
\begin{tikzpicture}
\begin{scope}
    \clip (0,0.25) rectangle (1.5,2.5);
    \draw[ultra thick] 
        (0,1.7) to[out=0,in=180]  (1.5,1.7);
    \draw[line width = 3pt, color=white]
        (0,1.3) to[out=0,in=90] (0.75,0.25) to[out=90,in=180] (1.5,1.3);
    \draw[line width = 3pt, color=white]
        (0.75,0.25) to[out=90,in=180] (1.5,1.7);
    \draw[ultra thick] 
        (0,1.3) to[out=0,in=90] (0.75,0.25) to[out=90,in=180] (1.5,1.3);
    \draw[ultra thick] 
    (0.75,0.25) to[out=90,in=180] (1.5,1.7);
\end{scope}
\draw[ultra thick,color=white!25!black,dashed] (0,0.25)--(1.5,0.25);
\draw[ultra thick,color=white!25!black,fill=white] (0.75,0.25) circle (3pt);
\draw[rounded corners] (0,0) rectangle (1.5,2.5);

\node[inner sep = 1pt, circle, draw, fill=white, ultra thick] at (0,1.3) {\tiny $2$};
\node[inner sep = 1pt, circle, draw, fill=white, ultra thick] at (0,1.7) {\tiny $1$};
\node[inner sep = 1pt, circle, draw, fill=white, ultra thick] at (1.5,1.3) {\tiny $1$};
\node[inner sep = 1pt, circle, draw, fill=white, ultra thick] at (1.5,1.7) {\tiny $2$};
\end{tikzpicture}}

\def\picEXd{
\begin{tikzpicture}
\begin{scope}
    \clip (0,0.25) rectangle (1.5,2.5);
    \draw[ultra thick] 
        (0,1.7) to[out=0,in=180]  (1.5,1.3);
    \draw[line width = 3pt, color=white]
        (0,1.3) to[out=0,in=90] (0.75,0.25) to[out=90,in=180] (1.5,1.7);
    \draw[ultra thick] 
        (0,1.3) to[out=0,in=90] (0.75,0.25) to[out=90,in=180] (1.5,1.7);
\end{scope}
\draw[ultra thick,color=white!25!black,dashed] (0,0.25)--(1.5,0.25);
\draw[ultra thick,color=white!25!black,fill=white] (0.75,0.25) circle (3pt);
\draw[rounded corners] (0,0) rectangle (1.5,2.5);

\node[inner sep = 1pt, circle, draw, fill=white, ultra thick] at (0,1.3) {\tiny $2$};
\node[inner sep = 1pt, circle, draw, fill=white, ultra thick] at (0,1.7) {\tiny $1$};
\node[inner sep = 1pt, circle, draw, fill=white, ultra thick] at (1.5,1.3) {\tiny $1$};
\node[inner sep = 1pt, circle, draw, fill=white, ultra thick] at (1.5,1.7) {\tiny $2$};
\end{tikzpicture}}

\def\picEXe{
\begin{tikzpicture}
\begin{scope}
    \clip (0,0.25) rectangle (1.5,2.5);
    \draw[ultra thick] 
        (0,1.7) to[out=0,in=180]  (1.5,2);
    \draw[line width = 3pt, color=white]
        (0,1.3) to[out=0,in=90] (0.75,0.25) to[out=90,in=180] (1.5,1.5);
    \draw[line width = 3pt, color=white]
        (0,1.3) to[out=0,in=90] (0.75,0.25) to[out=90,in=180] (1.5,1);
    \draw[ultra thick] 
        (0,1.3) to[out=0,in=90] (0.75,0.25) to[out=90,in=180] (1.5,1.5);
    \draw[ultra thick] 
        (0,1.3) to[out=0,in=90] (0.75,0.25) to[out=90,in=180] (1.5,1);
\end{scope}
\draw[ultra thick,color=white!25!black,dashed] (0,0.25)--(1.5,0.25);
\draw[ultra thick,color=white!25!black,fill=white] (0.75,0.25) circle (3pt);
\draw[rounded corners] (0,0) rectangle (1.5,2.5);

\node[inner sep = 1pt, circle, draw, fill=white, ultra thick] at (0,1.3) {\tiny $2$};
\node[inner sep = 1pt, circle, draw, fill=white, ultra thick] at (0,1.7) {\tiny $1$};
\node[inner sep = 1pt, circle, draw, fill=white, ultra thick] at (1.5,1) {\tiny $1$};
\node[inner sep = 1pt, circle, draw, fill=white, ultra thick] at (1.5,1.5) {\tiny $1$};
\node[inner sep = 1pt, circle, draw, fill=white, ultra thick] at (1.5,2) {\tiny $1$};
\end{tikzpicture}}

\def\picEXf{
\begin{tikzpicture}
\begin{scope}
    \clip (0,0.25) rectangle (1.5,2.5);
    \draw[ultra thick] 
        (0,1.7) to[out=0,in=180]  (1.5,1.5);
    \draw[line width = 3pt, color=white]
        (0,1.3) to[out=0,in=90] (0.75,0.25) to[out=90,in=180] (1.5,2);
    \draw[line width = 3pt, color=white]
        (0,1.3) to[out=0,in=90] (0.75,0.25) to[out=90,in=180] (1.5,1);
    \draw[ultra thick] 
        (0,1.3) to[out=0,in=90] (0.75,0.25) to[out=90,in=180] (1.5,2);
    \draw[ultra thick] 
        (0,1.3) to[out=0,in=90] (0.75,0.25) to[out=90,in=180] (1.5,1);
\end{scope}
\draw[ultra thick,color=white!25!black,dashed] (0,0.25)--(1.5,0.25);
\draw[ultra thick,color=white!25!black,fill=white] (0.75,0.25) circle (3pt);
\draw[rounded corners] (0,0) rectangle (1.5,2.5);

\node[inner sep = 1pt, circle, draw, fill=white, ultra thick] at (0,1.3) {\tiny $2$};
\node[inner sep = 1pt, circle, draw, fill=white, ultra thick] at (0,1.7) {\tiny $1$};
\node[inner sep = 1pt, circle, draw, fill=white, ultra thick] at (1.5,1) {\tiny $1$};
\node[inner sep = 1pt, circle, draw, fill=white, ultra thick] at (1.5,1.5) {\tiny $1$};
\node[inner sep = 1pt, circle, draw, fill=white, ultra thick] at (1.5,2) {\tiny $1$};
\end{tikzpicture}}

\def\picEXg{
\begin{tikzpicture}
\begin{scope}
    \clip (0,0.25) rectangle (1.5,2.5);
    \draw[ultra thick] 
        (0,1.7) to[out=0,in=180]  (1.5,1);
    \draw[line width = 3pt, color=white]
        (0,1.3) to[out=0,in=90] (0.75,0.25) to[out=90,in=180] (1.5,2);
    \draw[line width = 3pt, color=white]
        (0,1.3) to[out=0,in=90] (0.75,0.25) to[out=90,in=180] (1.5,1.5);
    \draw[ultra thick] 
        (0,1.3) to[out=0,in=90] (0.75,0.25) to[out=90,in=180] (1.5,2);
    \draw[ultra thick] 
        (0,1.3) to[out=0,in=90] (0.75,0.25) to[out=90,in=180] (1.5,1.5);
\end{scope}
\draw[ultra thick,color=white!25!black,dashed] (0,0.25)--(1.5,0.25);
\draw[ultra thick,color=white!25!black,fill=white] (0.75,0.25) circle (3pt);
\draw[rounded corners] (0,0) rectangle (1.5,2.5);

\node[inner sep = 1pt, circle, draw, fill=white, ultra thick] at (0,1.3) {\tiny $2$};
\node[inner sep = 1pt, circle, draw, fill=white, ultra thick] at (0,1.7) {\tiny $1$};
\node[inner sep = 1pt, circle, draw, fill=white, ultra thick] at (1.5,1) {\tiny $1$};
\node[inner sep = 1pt, circle, draw, fill=white, ultra thick] at (1.5,1.5) {\tiny $1$};
\node[inner sep = 1pt, circle, draw, fill=white, ultra thick] at (1.5,2) {\tiny $1$};
\end{tikzpicture}}

\def\picEXh{
\begin{tikzpicture}
\begin{scope}
    \clip (0,0.25) rectangle (1.5,2.5);
    \draw[ultra thick] 
        (0,1.7) to[out=0,in=180]  (1.5,1.5);
    \draw[line width = 3pt, color=white]
        (0,1.3) to[out=0,in=90] (0.75,0.25) to[out=90,in=180] (1.5,1.5);
    \draw[ultra thick] 
        (0,1.3) to[out=0,in=90] (0.75,0.25) to[out=90,in=180] (1.5,1.5);
\end{scope}
\draw[ultra thick,color=white!25!black,dashed] (0,0.25)--(1.5,0.25);
\draw[ultra thick,color=white!25!black,fill=white] (0.75,0.25) circle (3pt);
\draw[rounded corners] (0,0) rectangle (1.5,2.5);

\node[inner sep = 1pt, circle, draw, fill=white, ultra thick] at (0,1.3) {\tiny $2$};
\node[inner sep = 1pt, circle, draw, fill=white, ultra thick] at (0,1.7) {\tiny $1$};
\node[inner sep = 1pt, circle, draw, fill=white, ultra thick] at (1.5,1.5) {\tiny $3$};
\end{tikzpicture}}

\def\partpicA{
\begin{tikzpicture}[scale=0.4]
\fill[color=black!50!white,rounded corners] (0,0) rectangle (1,1) (0,1) rectangle (1,2) (0,2) rectangle (1,3);
\draw[thick,rounded corners] (0,0) rectangle (1,1) (0,1) rectangle (1,2) (0,2) rectangle (1,3);
\end{tikzpicture}}

\def\partpicB{
\begin{tikzpicture}[scale=0.4]
\fill[color=black!50!white,rounded corners] (0,1) rectangle (1,2) (0,2) rectangle (1,3);
\draw[thick,rounded corners] (0,0) rectangle (1,1) (0,1) rectangle (1,2) (0,2) rectangle (1,3);
\end{tikzpicture}}

\def\partpicC{
\begin{tikzpicture}[scale=0.4]
\fill[color=black!50!white,rounded corners] (0,0) rectangle (1,1) (0,2) rectangle (1,3);
\draw[thick,rounded corners] (0,0) rectangle (1,1) (0,1) rectangle (1,2) (0,2) rectangle (1,3);
\end{tikzpicture}}

\def\partpicD{
\begin{tikzpicture}[scale=0.4]
\fill[color=black!50!white,rounded corners] (0,0) rectangle (1,1) (0,1) rectangle (1,2);
\draw[thick,rounded corners] (0,0) rectangle (1,1) (0,1) rectangle (1,2) (0,2) rectangle (1,3);
\end{tikzpicture}}

\def\partpicE{
\begin{tikzpicture}[scale=0.4]
\fill[color=black!50!white,rounded corners] (0,0.5) rectangle (1,1.5) (0,1.5) rectangle (1,2.5);
\draw[thick,rounded corners] (0,0.5) rectangle (1,1.5) (1,0.5) rectangle (2,1.5) (0,1.5) rectangle (1,2.5);
\end{tikzpicture}}

\def\partpicF{
\begin{tikzpicture}[scale=0.4]
\fill[color=black!50!white,rounded corners] (0,0.5) rectangle (1,1.5) (0,1.5) rectangle (1,2.5);
\draw[thick,rounded corners] (0,0.5) rectangle (1,1.5) (0,1.5) rectangle (1,2.5) (1,1.5) rectangle (2,2.5);
\end{tikzpicture}}

\def\partpicG{
\begin{tikzpicture}[scale=0.4]
\fill[color=black!50!white,rounded corners] (0,0) rectangle (1,1);
\draw[thick,rounded corners] (0,0) rectangle (1,1) (0,1) rectangle (1,2) (0,2) rectangle (1,3);
\end{tikzpicture}}

\def\partpicH{
\begin{tikzpicture}[scale=0.4]
\fill[color=black!50!white,rounded corners] (0,1) rectangle (1,2);
\draw[thick,rounded corners] (0,0) rectangle (1,1) (0,1) rectangle (1,2) (0,2) rectangle (1,3);
\end{tikzpicture}}

\def\partpicI{
\begin{tikzpicture}[scale=0.4]
\fill[color=black!50!white,rounded corners] (0,2) rectangle (1,3);
\draw[thick,rounded corners] (0,0) rectangle (1,1) (0,1) rectangle (1,2) (0,2) rectangle (1,3);
\end{tikzpicture}}

\def\partpicJ{
\begin{tikzpicture}[scale=0.4]
\fill[color=black!50!white,rounded corners] (0,0.5) rectangle (1,1.5) (1,0.5) rectangle (2,1.5) (0,1.5) rectangle (1,2.5);
\draw[thick,rounded corners] (0,0.5) rectangle (1,1.5) (1,0.5) rectangle (2,1.5) (0,1.5) rectangle (1,2.5);
\end{tikzpicture}}

\def\partpicK{
\begin{tikzpicture}[scale=0.4]
\fill[color=black!50!white,rounded corners] (0,1.5) rectangle (1,2.5);
\draw[thick,rounded corners] (0,0.5) rectangle (1,1.5) (1,0.5) rectangle (2,1.5) (0,1.5) rectangle (1,2.5);
\end{tikzpicture}}

\def\partpicL{
\begin{tikzpicture}[scale=0.4]
\fill[color=black!50!white,rounded corners] (0,0.5) rectangle (1,1.5);
\draw[thick,rounded corners] (0,0.5) rectangle (1,1.5) (1,0.5) rectangle (2,1.5) (0,1.5) rectangle (1,2.5);
\end{tikzpicture}}

\def\partpicM{
\begin{tikzpicture}[scale=0.4]
\fill[color=black!50!white,rounded corners] (0,1) rectangle (1,2);
\draw[thick,rounded corners] (0,0) rectangle (1,1) (0,1) rectangle (1,2) (1,1) rectangle (2,2);
\end{tikzpicture}}

\def\partpicN{
\begin{tikzpicture}[scale=0.4]
\fill[color=black!50!white,rounded corners] (0,0) rectangle (1,1);
\draw[thick,rounded corners] (0,0) rectangle (1,1) (0,1) rectangle (1,2) (1,1) rectangle (2,2);
\end{tikzpicture}}

\def\partpicO{
\begin{tikzpicture}[scale=0.4]
\fill[color=black!50!white,rounded corners] (0,0.5) rectangle (1,1.5);
\draw[thick,rounded corners] (0,0.5) rectangle (1,1.5) (1,0.5) rectangle (2,1.5) (2,0.5) rectangle (3,1.5);
\end{tikzpicture}}

\def\partpicP{
\begin{tikzpicture}[scale=0.4]
\fill[color=black!50!white,rounded corners] (0,0) rectangle (1,1) (1,0) rectangle (2,1);
\draw[thick,rounded corners] (0,0) rectangle (1,1) (1,0) rectangle (2,1) (0,1) rectangle (1,2);
\end{tikzpicture}}

\def\partpicQ{
\begin{tikzpicture}[scale=0.4]
\fill[color=black!50!white,rounded corners] (0,0) rectangle (1,1) (0,1) rectangle (1,2) (1,1) rectangle (2,2);
\draw[thick,rounded corners] (0,0) rectangle (1,1) (0,1) rectangle (1,2) (1,1) rectangle (2,2);
\end{tikzpicture}}

\def\partpicR{
\begin{tikzpicture}[scale=0.4]
\fill[color=black!50!white,rounded corners] (0,1) rectangle (1,2) (1,1) rectangle (2,2);
\draw[thick,rounded corners] (0,0) rectangle (1,1) (0,1) rectangle (1,2) (1,1) rectangle (2,2);
\end{tikzpicture}}

\def\partpicS{
\begin{tikzpicture}[scale=0.4]
\fill[color=black!50!white,rounded corners] (0,0.5) rectangle (1,1.5) (1,0.5) rectangle (2,1.5);
\draw[thick,rounded corners] (0,0.5) rectangle (1,1.5) (1,0.5) rectangle (2,1.5) (2,0.5) rectangle (3,1.5);
\end{tikzpicture}}

\def\partpicT{
\begin{tikzpicture}[scale=0.4]
\fill[color=black!50!white,rounded corners];
\draw[thick,rounded corners] (0,0) rectangle (1,1) (0,1) rectangle (1,2) (0,2) rectangle (1,3);
\end{tikzpicture}}

\def\partpicU{
\begin{tikzpicture}[scale=0.4]
\fill[color=black!50!white,rounded corners];
\draw[thick,rounded corners] (0,0.5) rectangle (1,1.5) (1,0.5) rectangle (2,1.5) (0,1.5) rectangle (1,2.5);
\end{tikzpicture}}

\def\partpicV{
\begin{tikzpicture}[scale=0.4]
\fill[color=black!50!white,rounded corners];
\draw[thick,rounded corners] (0,0.5) rectangle (1,1.5) (0,1.5) rectangle (1,2.5) (1,1.5) rectangle (2,2.5);
\end{tikzpicture}}

\def\partpicW{
\begin{tikzpicture}[scale=0.4]
\fill[color=black!50!white,rounded corners] (0,1) rectangle (1,2) (1,1) rectangle (2,2) (2,1) rectangle (3,2);
\draw[thick,rounded corners] (0,1) rectangle (1,2) (1,1) rectangle (2,2) (2,1) rectangle (3,2);
\end{tikzpicture}}

\def\partpicX{
\begin{tikzpicture}[scale=0.4]
\fill[color=black!50!white,rounded corners];
\draw[thick,rounded corners] (0,1) rectangle (1,2) (1,1) rectangle (2,2) (2,1) rectangle (3,2);
\end{tikzpicture}}

\def\picJPaa{
\begin{tikzpicture}
\begin{scope}
    \clip (0,0.25) rectangle (1.5,2.5);
    \draw[ultra thick] 
        (0,1) to[out=0,in=180]  (1.5,1);
    \draw[ultra thick] 
        (0,1.5) to[out=0,in=180]  (1.5,1.5);
    \draw[ultra thick] 
        (0,2) to[out=0,in=180]  (1.5,2);
\end{scope}
\draw[ultra thick,color=white!25!black,dashed] (0,0.25)--(1.5,0.25);
\draw[ultra thick,color=white!25!black,fill=white] (0.75,0.25) circle (3pt);
\draw[rounded corners] (0,0) rectangle (1.5,2.5);

\node[inner sep = 1pt, circle, draw, fill=white, ultra thick] at (0,1) {\tiny $1$};
\node[inner sep = 1pt, circle, draw, fill=white, ultra thick] at (0,1.5) {\tiny $1$};
\node[inner sep = 1pt, circle, draw, fill=white, ultra thick] at (0,2) {\tiny $1$};
\node[inner sep = 1pt, circle, draw, fill=white, ultra thick] at (1.5,1) {\tiny $1$};
\node[inner sep = 1pt, circle, draw, fill=white, ultra thick] at (1.5,1.5) {\tiny $1$};
\node[inner sep = 1pt, circle, draw, fill=white, ultra thick] at (1.5,2) {\tiny $1$};
\end{tikzpicture}}

\def\picJPab{
\begin{tikzpicture}
\begin{scope}
    \clip (0,0.25) rectangle (1.5,2.5);
    \draw[line width = 3pt, color=white]
        (0,1) to[out=0,in=90] (0.75,0.25) to[out=90,in=180] (1.5,1);
    \draw[ultra thick] 
        (0,1) to[out=0,in=90] (0.75,0.25) to[out=90,in=180] (1.5,1);
    \draw[ultra thick] 
        (0,1.5) to[out=0,in=180]  (1.5,1.5);
    \draw[ultra thick] 
        (0,2) to[out=0,in=180]  (1.5,2);
\end{scope}
\draw[ultra thick,color=white!25!black,dashed] (0,0.25)--(1.5,0.25);
\draw[ultra thick,color=white!25!black,fill=white] (0.75,0.25) circle (3pt);
\draw[rounded corners] (0,0) rectangle (1.5,2.5);

\node[inner sep = 1pt, circle, draw, fill=white, ultra thick] at (0,1) {\tiny $1$};
\node[inner sep = 1pt, circle, draw, fill=white, ultra thick] at (0,1.5) {\tiny $1$};
\node[inner sep = 1pt, circle, draw, fill=white, ultra thick] at (0,2) {\tiny $1$};
\node[inner sep = 1pt, circle, draw, fill=white, ultra thick] at (1.5,1) {\tiny $1$};
\node[inner sep = 1pt, circle, draw, fill=white, ultra thick] at (1.5,1.5) {\tiny $1$};
\node[inner sep = 1pt, circle, draw, fill=white, ultra thick] at (1.5,2) {\tiny $1$};
\end{tikzpicture}}

\def\picJPac{
\begin{tikzpicture}
\begin{scope}
    \clip (0,0.25) rectangle (1.5,2.5);
    
    \draw[ultra thick] 
        (0,1.5) to[out=0,in=180]  (1.5,1);
    \draw[ultra thick] 
        (0,2) to[out=0,in=180]  (1.5,2);
    \draw[line width = 3pt, color=white]
        (0,1) to[out=0,in=90] (0.75,0.25) to[out=90,in=180] (1.5,1.5);
    \draw[ultra thick] 
        (0,1) to[out=0,in=90] (0.75,0.25) to[out=90,in=180] (1.5,1.5);
\end{scope}
\draw[ultra thick,color=white!25!black,dashed] (0,0.25)--(1.5,0.25);
\draw[ultra thick,color=white!25!black,fill=white] (0.75,0.25) circle (3pt);
\draw[rounded corners] (0,0) rectangle (1.5,2.5);

\node[inner sep = 1pt, circle, draw, fill=white, ultra thick] at (0,1) {\tiny $1$};
\node[inner sep = 1pt, circle, draw, fill=white, ultra thick] at (0,1.5) {\tiny $1$};
\node[inner sep = 1pt, circle, draw, fill=white, ultra thick] at (0,2) {\tiny $1$};
\node[inner sep = 1pt, circle, draw, fill=white, ultra thick] at (1.5,1) {\tiny $1$};
\node[inner sep = 1pt, circle, draw, fill=white, ultra thick] at (1.5,1.5) {\tiny $1$};
\node[inner sep = 1pt, circle, draw, fill=white, ultra thick] at (1.5,2) {\tiny $1$};
\end{tikzpicture}}

\def\picJPad{
\begin{tikzpicture}
\begin{scope}
    \clip (0,0.25) rectangle (1.5,2.5);
    
    \draw[ultra thick] 
        (0,1.5) to[out=0,in=180]  (1.5,1);
    \draw[ultra thick] 
        (0,2) to[out=0,in=180]  (1.5,1.5);
    \draw[line width = 3pt, color=white]
        (0,1) to[out=0,in=90] (0.75,0.25) to[out=90,in=180] (1.5,2);
    \draw[ultra thick] 
        (0,1) to[out=0,in=90] (0.75,0.25) to[out=90,in=180] (1.5,2);
\end{scope}
\draw[ultra thick,color=white!25!black,dashed] (0,0.25)--(1.5,0.25);
\draw[ultra thick,color=white!25!black,fill=white] (0.75,0.25) circle (3pt);
\draw[rounded corners] (0,0) rectangle (1.5,2.5);

\node[inner sep = 1pt, circle, draw, fill=white, ultra thick] at (0,1) {\tiny $1$};
\node[inner sep = 1pt, circle, draw, fill=white, ultra thick] at (0,1.5) {\tiny $1$};
\node[inner sep = 1pt, circle, draw, fill=white, ultra thick] at (0,2) {\tiny $1$};
\node[inner sep = 1pt, circle, draw, fill=white, ultra thick] at (1.5,1) {\tiny $1$};
\node[inner sep = 1pt, circle, draw, fill=white, ultra thick] at (1.5,1.5) {\tiny $1$};
\node[inner sep = 1pt, circle, draw, fill=white, ultra thick] at (1.5,2) {\tiny $1$};
\end{tikzpicture}}

\def\picJPae{
\begin{tikzpicture}
\begin{scope}
    \clip (0,0.25) rectangle (1.5,2.5);
    
    \draw[ultra thick] 
        (0,1.5) to[out=0,in=180]  (1.5,1.3);
    \draw[ultra thick] 
        (0,2) to[out=0,in=180]  (1.5,1.7);
    \draw[line width = 3pt, color=white]
        (0,1) to[out=0,in=90] (0.75,0.25) to[out=90,in=180] (1.5,1.3);
    \draw[ultra thick] 
        (0,1) to[out=0,in=90] (0.75,0.25) to[out=90,in=180] (1.5,1.3);
\end{scope}
\draw[ultra thick,color=white!25!black,dashed] (0,0.25)--(1.5,0.25);
\draw[ultra thick,color=white!25!black,fill=white] (0.75,0.25) circle (3pt);
\draw[rounded corners] (0,0) rectangle (1.5,2.5);

\node[inner sep = 1pt, circle, draw, fill=white, ultra thick] at (0,1) {\tiny $1$};
\node[inner sep = 1pt, circle, draw, fill=white, ultra thick] at (0,1.5) {\tiny $1$};
\node[inner sep = 1pt, circle, draw, fill=white, ultra thick] at (0,2) {\tiny $1$};
\node[inner sep = 1pt, circle, draw, fill=white, ultra thick] at (1.5,1.3) {\tiny $2$};
\node[inner sep = 1pt, circle, draw, fill=white, ultra thick] at (1.5,1.7) {\tiny $1$};

\end{tikzpicture}}

\def\picJPaf{
\begin{tikzpicture}
\begin{scope}
    \clip (0,0.25) rectangle (1.5,2.5);
    
    \draw[ultra thick] 
        (0,1.5) to[out=0,in=180]  (1.5,1.3);
    \draw[ultra thick] 
        (0,2) to[out=0,in=180]  (1.5,1.7);
    \draw[line width = 3pt, color=white]
        (0,1) to[out=0,in=90] (0.75,0.25) to[out=90,in=180] (1.5,1.7);
    \draw[ultra thick] 
        (0,1) to[out=0,in=90] (0.75,0.25) to[out=90,in=180] (1.5,1.7);
\end{scope}
\draw[ultra thick,color=white!25!black,dashed] (0,0.25)--(1.5,0.25);
\draw[ultra thick,color=white!25!black,fill=white] (0.75,0.25) circle (3pt);
\draw[rounded corners] (0,0) rectangle (1.5,2.5);

\node[inner sep = 1pt, circle, draw, fill=white, ultra thick] at (0,1) {\tiny $1$};
\node[inner sep = 1pt, circle, draw, fill=white, ultra thick] at (0,1.5) {\tiny $1$};
\node[inner sep = 1pt, circle, draw, fill=white, ultra thick] at (0,2) {\tiny $1$};
\node[inner sep = 1pt, circle, draw, fill=white, ultra thick] at (1.5,1.7) {\tiny $2$};
\node[inner sep = 1pt, circle, draw, fill=white, ultra thick] at (1.5,1.3) {\tiny $1$};

\end{tikzpicture}}

\def\picJPbd{
\begin{tikzpicture}
\begin{scope}
    \clip (0,0.25) rectangle (1.5,2.5);
    
    \draw[ultra thick] 
        (0,1.3) to[out=0,in=180]  (1.5,1.3);
    \draw[ultra thick] 
        (0,1.7) to[out=0,in=180]  (1.5,1.7);
    
\end{scope}
\draw[ultra thick,color=white!25!black,dashed] (0,0.25)--(1.5,0.25);
\draw[ultra thick,color=white!25!black,fill=white] (0.75,0.25) circle (3pt);
\draw[rounded corners] (0,0) rectangle (1.5,2.5);

\node[inner sep = 1pt, circle, draw, fill=white, ultra thick] at (0,1.3) {\tiny $2$};
\node[inner sep = 1pt, circle, draw, fill=white, ultra thick] at (0,1.7) {\tiny $1$};

\node[inner sep = 1pt, circle, draw, fill=white, ultra thick] at (1.5,1.3) {\tiny $2$};
\node[inner sep = 1pt, circle, draw, fill=white, ultra thick] at (1.5,1.7) {\tiny $1$};

\end{tikzpicture}}

\def\picJPcd{
\begin{tikzpicture}
\begin{scope}
    \clip (0,0.25) rectangle (1.5,2.5);

    \draw[ultra thick] 
        (0,1.7) to[out=0,in=180]  (1.5,1.3);
    \draw[line width = 3pt, color=white]
        (0,1.3) to[out=0,in=90] (0.75,0.25) to[out=90,in=180] (1.5,1.7);
    \draw[ultra thick] 
        (0,1.3) to[out=0,in=90] (0.75,0.25) to[out=90,in=180] (1.5,1.7);
\end{scope}
\draw[ultra thick,color=white!25!black,dashed] (0,0.25)--(1.5,0.25);
\draw[ultra thick,color=white!25!black,fill=white] (0.75,0.25) circle (3pt);
\draw[rounded corners] (0,0) rectangle (1.5,2.5);

\node[inner sep = 1pt, circle, draw, fill=white, ultra thick] at (0,1.3) {\tiny $1$};
\node[inner sep = 1pt, circle, draw, fill=white, ultra thick] at (0,1.7) {\tiny $2$};

\node[inner sep = 1pt, circle, draw, fill=white, ultra thick] at (1.5,1.3) {\tiny $2$};
\node[inner sep = 1pt, circle, draw, fill=white, ultra thick] at (1.5,1.7) {\tiny $1$};

\end{tikzpicture}}

\def\picJPce{
\begin{tikzpicture}
\begin{scope}
    \clip (0,0.25) rectangle (1.5,2.5);
    
    \draw[ultra thick] 
        (0,1.3) to[out=0,in=180]  (1.5,1.3);
    \draw[ultra thick] 
        (0,1.7) to[out=0,in=180]  (1.5,1.7);
    
\end{scope}
\draw[ultra thick,color=white!25!black,dashed] (0,0.25)--(1.5,0.25);
\draw[ultra thick,color=white!25!black,fill=white] (0.75,0.25) circle (3pt);
\draw[rounded corners] (0,0) rectangle (1.5,2.5);

\node[inner sep = 1pt, circle, draw, fill=white, ultra thick] at (0,1.3) {\tiny $1$};
\node[inner sep = 1pt, circle, draw, fill=white, ultra thick] at (0,1.7) {\tiny $2$};

\node[inner sep = 1pt, circle, draw, fill=white, ultra thick] at (1.5,1.3) {\tiny $1$};
\node[inner sep = 1pt, circle, draw, fill=white, ultra thick] at (1.5,1.7) {\tiny $2$};

\end{tikzpicture}}

\def\picJPcf{
\begin{tikzpicture}
\begin{scope}
    \clip (0,0.25) rectangle (1.5,2.5);

    \draw[ultra thick] 
        (0,1.7) to[out=0,in=180]  (1.5,1.7);
    \draw[line width = 3pt, color=white]
        (0,1.3) to[out=0,in=90] (0.75,0.25) to[out=90,in=180] (1.5,1.3);
    \draw[ultra thick] 
        (0,1.3) to[out=0,in=90] (0.75,0.25) to[out=90,in=180] (1.5,1.3);
\end{scope}
\draw[ultra thick,color=white!25!black,dashed] (0,0.25)--(1.5,0.25);
\draw[ultra thick,color=white!25!black,fill=white] (0.75,0.25) circle (3pt);
\draw[rounded corners] (0,0) rectangle (1.5,2.5);

\node[inner sep = 1pt, circle, draw, fill=white, ultra thick] at (0,1.3) {\tiny $1$};
\node[inner sep = 1pt, circle, draw, fill=white, ultra thick] at (0,1.7) {\tiny $2$};

\node[inner sep = 1pt, circle, draw, fill=white, ultra thick] at (1.5,1.3) {\tiny $1$};
\node[inner sep = 1pt, circle, draw, fill=white, ultra thick] at (1.5,1.7) {\tiny $2$};

\end{tikzpicture}}

\def\picJPda{
\begin{tikzpicture}
\begin{scope}
    \clip (0,0.25) rectangle (1.5,2.5);
    \draw[ultra thick] 
        (0,1.7) to[out=0,in=180]  (1.5,1.5);
    \draw[line width = 3pt, color=white]
        (0,1.3) to[out=0,in=90] (0.75,0.25) to[out=90,in=180] (1.5,1.5);
    \draw[ultra thick] 
        (0,1.3) to[out=0,in=90] (0.75,0.25) to[out=90,in=180] (1.5,1.5);
\end{scope}
\draw[ultra thick,color=white!25!black,dashed] (0,0.25)--(1.5,0.25);
\draw[ultra thick,color=white!25!black,fill=white] (0.75,0.25) circle (3pt);
\draw[rounded corners] (0,0) rectangle (1.5,2.5);

\node[inner sep = 1pt, circle, draw, fill=white, ultra thick] at (0,1.3) {\tiny $1$};
\node[inner sep = 1pt, circle, draw, fill=white, ultra thick] at (0,1.7) {\tiny $2$};

\node[inner sep = 1pt, circle, draw, fill=white, ultra thick] at (1.5,1.5) {\tiny $3$};
\end{tikzpicture}}

\def\picJPdb{
\begin{tikzpicture}
\begin{scope}
    \clip (0,0.25) rectangle (1.5,2.5);
    \draw[ultra thick] 
        (0,1.5) to[out=0,in=90] (0.75,0.25) to[out=90,in=180] (1.5,1);
    \draw[ultra thick] 
        (0,1.5) to[out=0,in=90] (0.75,0.25) to[out=90,in=180] (1.5,1.5);
    \draw[ultra thick] 
        (0,1.5) to[out=0,in=90] (0.75,0.25) to[out=90,in=180] (1.5,2);
\end{scope}
\draw[ultra thick,color=white!25!black,dashed] (0,0.25)--(1.5,0.25);
\draw[ultra thick,color=white!25!black,fill=white] (0.75,0.25) circle (3pt);
\draw[rounded corners] (0,0) rectangle (1.5,2.5);

\node[inner sep = 1pt, circle, draw, fill=white, ultra thick] at (0,1.5) {\tiny $3$};
\node[inner sep = 1pt, circle, draw, fill=white, ultra thick] at (1.5,1) {\tiny $1$};
\node[inner sep = 1pt, circle, draw, fill=white, ultra thick] at (1.5,2) {\tiny $1$};
\node[inner sep = 1pt, circle, draw, fill=white, ultra thick] at (1.5,1.5) {\tiny $1$};
\end{tikzpicture}}

\def\picJPdc{
\begin{tikzpicture}
\begin{scope}
    \clip (0,0.25) rectangle (1.5,2.5);
    \draw[ultra thick] 
        (0,1.5) to[out=0,in=90] (0.75,0.25) to[out=90,in=180] (1.5,1.3);
    \draw[ultra thick] 
        (0,1.5) to[out=0,in=90] (0.75,0.25) to[out=90,in=180] (1.5,1.7);
\end{scope}
\draw[ultra thick,color=white!25!black,dashed] (0,0.25)--(1.5,0.25);
\draw[ultra thick,color=white!25!black,fill=white] (0.75,0.25) circle (3pt);
\draw[rounded corners] (0,0) rectangle (1.5,2.5);

\node[inner sep = 1pt, circle, draw, fill=white, ultra thick] at (0,1.5) {\tiny $3$};
\node[inner sep = 1pt, circle, draw, fill=white, ultra thick] at (1.5,1.3) {\tiny $2$};
\node[inner sep = 1pt, circle, draw, fill=white, ultra thick] at (1.5,1.7) {\tiny $1$};
\end{tikzpicture}}

\def\picJPdd{
\begin{tikzpicture}
\begin{scope}
    \clip (0,0.25) rectangle (1.5,2.5);
    \draw[ultra thick] 
        (0,1.5) to[out=0,in=90] (0.75,0.25) to[out=90,in=180] (1.5,1.3);
    \draw[ultra thick] 
        (0,1.5) to[out=0,in=90] (0.75,0.25) to[out=90,in=180] (1.5,1.7);
\end{scope}
\draw[ultra thick,color=white!25!black,dashed] (0,0.25)--(1.5,0.25);
\draw[ultra thick,color=white!25!black,fill=white] (0.75,0.25) circle (3pt);
\draw[rounded corners] (0,0) rectangle (1.5,2.5);

\node[inner sep = 1pt, circle, draw, fill=white, ultra thick] at (0,1.5) {\tiny $3$};
\node[inner sep = 1pt, circle, draw, fill=white, ultra thick] at (1.5,1.3) {\tiny $1$};
\node[inner sep = 1pt, circle, draw, fill=white, ultra thick] at (1.5,1.7) {\tiny $2$};
\end{tikzpicture}}

\def\picJPde{
\begin{tikzpicture}
\begin{scope}
    \clip (0,0.25) rectangle (1.5,2.5);
    \draw[ultra thick] 
        (0,1.5) to[out=0,in=180] (1.5,1.5);
\end{scope}
\draw[ultra thick,color=white!25!black,dashed] (0,0.25)--(1.5,0.25);
\draw[ultra thick,color=white!25!black,fill=white] (0.75,0.25) circle (3pt);
\draw[rounded corners] (0,0) rectangle (1.5,2.5);

\node[inner sep = 1pt, circle, draw, fill=white, ultra thick] at (0,1.5) {\tiny $3$};
\node[inner sep = 1pt, circle, draw, fill=white, ultra thick] at (1.5,1.5) {\tiny $3$};
\end{tikzpicture}}

\def\picJPdf{
\begin{tikzpicture}
\begin{scope}
    \clip (0,0.25) rectangle (1.5,2.5);
    \draw[ultra thick] 
        (0,1.5) to[out=0,in=90] (0.75,0.25) to[out=90,in=180] (1.5,1.5);
\end{scope}
\draw[ultra thick,color=white!25!black,dashed] (0,0.25)--(1.5,0.25);
\draw[ultra thick,color=white!25!black,fill=white] (0.75,0.25) circle (3pt);
\draw[rounded corners] (0,0) rectangle (1.5,2.5);

\node[inner sep = 1pt, circle, draw, fill=white, ultra thick] at (0,1.5) {\tiny $3$};
\node[inner sep = 1pt, circle, draw, fill=white, ultra thick] at (1.5,1.5) {\tiny $3$};
\end{tikzpicture}}

\def\picWeda{
\begin{tikzpicture}
\begin{scope}
    \clip (0,0.25) rectangle (1.5,2.5);
    \draw[ultra thick] 
        (0,2.2) to[out=0,in=180] (1.5,2.2);
    \draw[ultra thick] 
        (0,1.8) to[out=0,in=180] (1.5,1.8);
    \draw[ultra thick] 
        (0,1.4) to[out=0,in=180] (1.5,1.2);
    \draw[ultra thick] 
        (0,1) to[out=0,in=90] (0.75,0.25) to[out=90,in=180] (1.5,1.2);
\end{scope}
\draw[ultra thick,color=white!25!black,dashed] (0,0.25)--(1.5,0.25);
\draw[ultra thick,color=white!25!black,fill=white] (0.75,0.25) circle (3pt);
\draw[rounded corners] (0,0) rectangle (1.5,2.5);

\node[inner sep = 1pt, circle, draw, fill=white, ultra thick] at (0,2.2) {\tiny $1$};
\node[inner sep = 1pt, circle, draw, fill=white, ultra thick] at (0,1.8) {\tiny $1$};
\node[inner sep = 1pt, circle, draw, fill=white, ultra thick] at (0,1.4) {\tiny $2$};
\node[inner sep = 1pt, circle, draw, fill=white, ultra thick] at (0,1) {\tiny $1$};
\node[inner sep = 1pt, circle, draw, fill=white, ultra thick] at (1.5,1.2) {\tiny $3$};
\node[inner sep = 1pt, circle, draw, fill=white, ultra thick] at (1.5,1.8) {\tiny $1$};
\node[inner sep = 1pt, circle, draw, fill=white, ultra thick] at (1.5,2.2) {\tiny $1$};
\end{tikzpicture}}

\def\picWedb{
\begin{tikzpicture}
\begin{scope}
    \clip (0,0.25) rectangle (1.5,2.5);
    \draw[ultra thick] 
        (0,2.2) to[out=0,in=180] (1.5,2.2);
    \draw[ultra thick] 
        (0,1.8) to[out=0,in=180] (1.5,1.8);
    \draw[ultra thick] 
        (0,1.2) to[out=0,in=90] (0.75,0.25) to[out=90,in=180] (1.5,1.4);
    \draw[ultra thick] 
        (0.75,0.25) to[out=90,in=180] (1.5,1);
    
\end{scope}
\draw[ultra thick,color=white!25!black,dashed] (0,0.25)--(1.5,0.25);
\draw[ultra thick,color=white!25!black,fill=white] (0.75,0.25) circle (3pt);
\draw[rounded corners] (0,0) rectangle (1.5,2.5);

\node[inner sep = 1pt, circle, draw, fill=white, ultra thick] at (0,2.2) {\tiny $1$};
\node[inner sep = 1pt, circle, draw, fill=white, ultra thick] at (0,1.8) {\tiny $1$};
\node[inner sep = 1pt, circle, draw, fill=white, ultra thick] at (0,1.2) {\tiny $3$};
\node[inner sep = 1pt, circle, draw, fill=white, ultra thick] at (1.5,1) {\tiny $2$};
\node[inner sep = 1pt, circle, draw, fill=white, ultra thick] at (1.5,1.4) {\tiny $1$};
\node[inner sep = 1pt, circle, draw, fill=white, ultra thick] at (1.5,1.8) {\tiny $1$};
\node[inner sep = 1pt, circle, draw, fill=white, ultra thick] at (1.5,2.2) {\tiny $1$};
\end{tikzpicture}}

\def\picWedc{
\begin{tikzpicture}
\begin{scope}
    \clip (0,0.25) rectangle (1.5,2.5);
    \draw[ultra thick] 
        (0,2.2) to[out=0,in=180] (1.5,2.2);
    \draw[ultra thick] 
        (0,1.8) to[out=0,in=180] (1.5,1.8);
    \draw[ultra thick] 
        (0,1.4) to[out=0,in=180] (1.5,1);
    \draw[line width = 4pt, color=white]
        (0,1) to[out=0,in=90] (0.75,0.25) to[out=90,in=180] (1.5,1.4);
    \draw[ultra thick] 
        (0,1) to[out=0,in=90] (0.75,0.25) to[out=90,in=180] (1.5,1.4);
   
\end{scope}
\draw[ultra thick,color=white!25!black,dashed] (0,0.25)--(1.5,0.25);
\draw[ultra thick,color=white!25!black,fill=white] (0.75,0.25) circle (3pt);
\draw[rounded corners] (0,0) rectangle (1.5,2.5);

\node[inner sep = 1pt, circle, draw, fill=white, ultra thick] at (0,2.2) {\tiny $1$};
\node[inner sep = 1pt, circle, draw, fill=white, ultra thick] at (0,1.8) {\tiny $1$};
\node[inner sep = 1pt, circle, draw, fill=white, ultra thick] at (0,1.4) {\tiny $1$};
\node[inner sep = 1pt, circle, draw, fill=white, ultra thick] at (0,1) {\tiny $2$};
\node[inner sep = 1pt, circle, draw, fill=white, ultra thick] at (1.5,1) {\tiny $1$};
\node[inner sep = 1pt, circle, draw, fill=white, ultra thick] at (1.5,1.4) {\tiny $2$};
\node[inner sep = 1pt, circle, draw, fill=white, ultra thick] at (1.5,1.8) {\tiny $1$};
\node[inner sep = 1pt, circle, draw, fill=white, ultra thick] at (1.5,2.2) {\tiny $1$};
\end{tikzpicture}}

\def\picWedd{
\begin{tikzpicture}
\begin{scope}
    \clip (0,0.25) rectangle (1.5,2.5);
    \draw[ultra thick] 
        (0,2) to[out=0,in=180] (1.5,2);
    \draw[ultra thick] 
        (0,1.4) to[out=0,in=180] (1.5,1.2);
    \draw[ultra thick] 
        (0,1) to[out=0,in=90] (0.75,0.25) to[out=90,in=180] (1.5,1.2);
\end{scope}
\draw[ultra thick,color=white!25!black,dashed] (0,0.25)--(1.5,0.25);
\draw[ultra thick,color=white!25!black,fill=white] (0.75,0.25) circle (3pt);
\draw[rounded corners] (0,0) rectangle (1.5,2.5);

\node[inner sep = 1pt, circle, draw, fill=white, ultra thick] at (0,2) {\tiny $2$};
\node[inner sep = 1pt, circle, draw, fill=white, ultra thick] at (0,1.4) {\tiny $2$};
\node[inner sep = 1pt, circle, draw, fill=white, ultra thick] at (0,1) {\tiny $1$};
\node[inner sep = 1pt, circle, draw, fill=white, ultra thick] at (1.5,1.2) {\tiny $3$};
\node[inner sep = 1pt, circle, draw, fill=white, ultra thick] at (1.5,2) {\tiny $2$};
\end{tikzpicture}}

\def\picWede{
\begin{tikzpicture}
\begin{scope}
    \clip (0,0.25) rectangle (1.5,2.5);
    \draw[ultra thick] 
        (0,2) to[out=0,in=180] (1.5,2);
    \draw[ultra thick] 
        (0,1.2) to[out=0,in=90] (0.75,0.25) to[out=90,in=180] (1.5,1.4);
    \draw[ultra thick] 
        (0.75,0.25) to[out=90,in=180] (1.5,1);
    
\end{scope}
\draw[ultra thick,color=white!25!black,dashed] (0,0.25)--(1.5,0.25);
\draw[ultra thick,color=white!25!black,fill=white] (0.75,0.25) circle (3pt);
\draw[rounded corners] (0,0) rectangle (1.5,2.5);

\node[inner sep = 1pt, circle, draw, fill=white, ultra thick] at (0,2) {\tiny $2$};
\node[inner sep = 1pt, circle, draw, fill=white, ultra thick] at (0,1.2) {\tiny $3$};
\node[inner sep = 1pt, circle, draw, fill=white, ultra thick] at (1.5,1) {\tiny $2$};
\node[inner sep = 1pt, circle, draw, fill=white, ultra thick] at (1.5,1.4) {\tiny $1$};
\node[inner sep = 1pt, circle, draw, fill=white, ultra thick] at (1.5,2) {\tiny $2$};
\end{tikzpicture}}

\def\picWedf{
\begin{tikzpicture}
\begin{scope}
    \clip (0,0.25) rectangle (1.5,2.5);
    \draw[ultra thick] 
        (0,2) to[out=0,in=180] (1.5,2);
    \draw[ultra thick] 
        (0,1.4) to[out=0,in=180] (1.5,1);
    \draw[line width = 4pt, color=white]
        (0,1) to[out=0,in=90] (0.75,0.25) to[out=90,in=180] (1.5,1.4);
    \draw[ultra thick] 
        (0,1) to[out=0,in=90] (0.75,0.25) to[out=90,in=180] (1.5,1.4);
   
\end{scope}
\draw[ultra thick,color=white!25!black,dashed] (0,0.25)--(1.5,0.25);
\draw[ultra thick,color=white!25!black,fill=white] (0.75,0.25) circle (3pt);
\draw[rounded corners] (0,0) rectangle (1.5,2.5);

\node[inner sep = 1pt, circle, draw, fill=white, ultra thick] at (0,2) {\tiny $2$};
\node[inner sep = 1pt, circle, draw, fill=white, ultra thick] at (0,1.4) {\tiny $1$};
\node[inner sep = 1pt, circle, draw, fill=white, ultra thick] at (0,1) {\tiny $2$};
\node[inner sep = 1pt, circle, draw, fill=white, ultra thick] at (1.5,1) {\tiny $1$};
\node[inner sep = 1pt, circle, draw, fill=white, ultra thick] at (1.5,1.4) {\tiny $2$};
\node[inner sep = 1pt, circle, draw, fill=white, ultra thick] at (1.5,2) {\tiny $2$};
\end{tikzpicture}}

\def\picThu{
\begin{tikzpicture}
\begin{scope}
    \clip (0,0.25) rectangle (1.5,2.5);
    \draw[ultra thick] 
        (0,2.2) to[out=0,in=180] (1.5,1.8);
    \draw[ultra thick] 
        (0,1.8) to[out=0,in=180] (1.5,1.4);
    \draw[ultra thick] 
        (0,1.4) to[out=0,in=180] (1.5,1);
    \draw[ultra thick] 
        (0,1) to[out=0,in=90] (0.75,0.25);
    \draw[line width = 4pt, color=white]
        (0.75,0.25) to[out=90,in=180] (1.5,1.4);
    \draw[ultra thick] 
        (0.75,0.25) to[out=90,in=180] (1.5,1.4);
    \draw[line width = 4pt, color=white]
        (0.75,0.25) to[out=90,in=180] (1.5,2.2);
    \draw[ultra thick] 
        (0.75,0.25) to[out=90,in=180] (1.5,2.2);
\end{scope}
\draw[ultra thick,color=white!25!black,dashed] (0,0.25)--(1.5,0.25);
\draw[ultra thick,color=white!25!black,fill=white] (0.75,0.25) circle (3pt);
\draw[rounded corners] (0,0) rectangle (1.5,2.5);

\node[inner sep = 1pt, circle, draw, fill=white, ultra thick] at (0,2.2) {\tiny $2$};
\node[inner sep = 1pt, circle, draw, fill=white, ultra thick] at (0,1.8) {\tiny $2$};
\node[inner sep = 1pt, circle, draw, fill=white, ultra thick] at (0,1.4) {\tiny $3$};
\node[inner sep = 1pt, circle, draw, fill=white, ultra thick] at (0,1) {\tiny $3$};
\node[inner sep = 1pt, circle, draw, fill=white, ultra thick] at (1.5,2.2) {\tiny $2$};
\node[inner sep = 1pt, circle, draw, fill=white, ultra thick] at (1.5,1.8) {\tiny $2$};
\node[inner sep = 1pt, circle, draw, fill=white, ultra thick] at (1.5,1.4) {\tiny $3$};
\node[inner sep = 1pt, circle, draw, fill=white, ultra thick] at (1.5,1) {\tiny $3$};
\end{tikzpicture}}

\DeclareMathOperator{\trace}{tr}

\newtheorem{theorem}{Theorem}

\newtheorem{lemma}{Lemma}

\newtheorem{observation}{Observation}
\newtheorem{conjecture}{Conjecture}
\theoremstyle{definition}

\newtheorem{definition}{Definition}

\title{Enumerating multiplex juggling patterns}
\author{Steve Butler\footnote{Dept. of Mathematics, Iowa State University, Ames, IA 50011, USA; email: {\tt butler@iastate.edu}} \and Jeongyoon Choi\footnote{Gyeonggi Science High School for the Gifted, Gyeonggi Province, Republic of Korea; email: {\tt ericggul@gmail.com} (Choi), {\tt rlaud99102@naver.com} (Kim), {\tt 99skh@naver.com} (Seo)} \and Kimyung Kim\footnotemark[2] \and Kyuhyeok Seo\footnotemark[2]}
\date{\empty}

\begin{document}
\maketitle

\begin{abstract}
Mathematics has been used in the exploration and enumeration of juggling patterns.  In the case when we catch and throw one ball at a time the number of possible juggling patterns is  well-known.  When we are allowed to catch and throw any number of balls at a given time (known as multiplex juggling) the enumeration is more difficult and has only been established in a few special cases.  We give a method of using cards related to ``embeddings'' of ordered partitions to enumerate the count of multiplex juggling sequences, determine these counts for small cases, and establish some combinatorial properties of the set of cards needed.
\end{abstract}

\section{Introduction}
While mathematics and juggling have existed independently for thousands of years, it has only been in the last thirty years that the mathematics of juggling has become a subject in its own right (for a general introduction see Polster \cite{P}).  Several different approaches for describing juggling patterns have been used.  The best-known method is siteswap which gives information what to do with the ball that is in your hand at the given moment, in particular how ``high'' you should throw the ball (see \cite{BEGW}).  For theoretical purposes a more useful method is to describe patterns by the use of cards.  This was first introduced in the work of Ehrenborg and Readdy \cite{ER}, and modified by Butler, Chung, Cummings and Graham \cite {BCCG}.

These cards work by focusing on looking at the \emph{relative order} of when the balls will land should we stop juggling at a given moment. Every throw then has the effect of changing the relative ordering of the balls. But we can only effect the order of a ball that is thrown; the remaining balls will still have the remaining relative order to each other. As a consequence if there are $b$ balls there are $b+1$ different things which can happen.  Namely, we don't throw a ball (the ``$+1$'') or we throw a ball so that it will land in some order relative to the other $b-1$ balls (which can be done in $b$ ways).  The four different cards for the case $b=3$ are shown in Figure~\ref{fig:BasicCards} (in all drawings of cards the circle at the bottom indicates the hand which either does not catch the ball at that ``beat'' or catches and throws effecting the relative ordering of the ball(s); we will always think of time moving from left to right).

\begin{figure}[htb!]
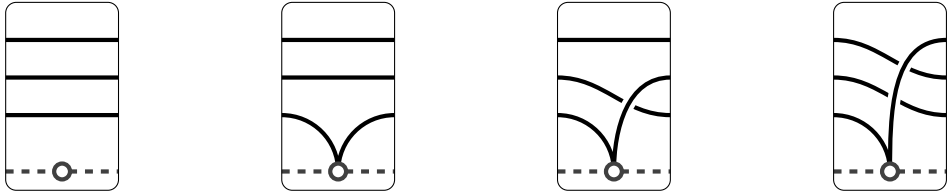

\centering
\picCa \hfil \picCb \hfil \picCc \hfil \picCd
\caption{Possible juggling cards for three balls.}
\label{fig:BasicCards}
\end{figure}

The advantage of working with cards is that the cards can work independently of each other, that is the choice of card to use at a given time is not dependent on the preceding choice of cards.  In siteswap the opposite is true in that you must know all preceding throws to determine which throws are possible.

Given a set of these $n$ cards for a given $b$ we can repeat these periodically to form a pattern.  Moreover, every possible siteswap with period $n$ and at most $b+1$ balls will occur as a unique combination of these cards (see \cite{BCCG,P}).  Therefore the number of different siteswap sequences of period $n$ for exactly $b$ balls is given by
\[
(b+1)^n-b^n.
\]
If we want to find all of the juggling patterns of \emph{minimal} period $n$ and using exactly $b$ balls we can then use M\"obius inversion and divide out by the period to get
\[
\frac1n\sum_{d\mid n}\mu({\textstyle\frac{n}{d}})\big((b+1)^d-b^d\big),
\]
where $\mu$ is the M\"obius function (see \cite{BEGW}).  (We will revisit this with more detail in Section~\ref{sec:counts}.)

For as long as there has been interest in the mathematics of juggling there has been interest in extending results to multiplex juggling (where more than one ball is allowed to be caught at a time).  In Ehrenborg and Readdy they produced possible cards for multiplex juggling which were a natural generalization.  Namely multiple balls could come down at a given time and would then be redistributed appropriately.  While these cards can describe every juggling pattern there is the problem that uniqueness is lost (see Figure~\ref{fig:nonunique} for an example of two consecutive cards describing the same pattern but using different cards).  So using these cards to count multiplex juggling patterns is not straight-forward.

\begin{figure}[htb!]
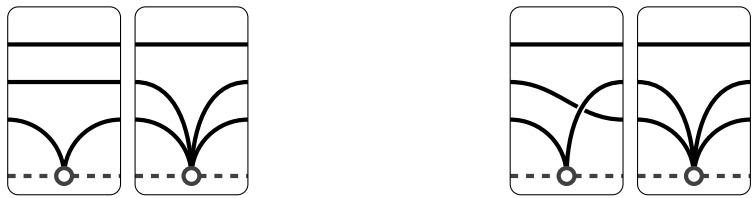

\centering
\picERc \picERa \hfil \picERd \picERa
\caption{Ambiguity arising from using cards of Ehrenborg and Readdy \cite{ER}.}
\label{fig:nonunique}
\end{figure}

One approach is to distinguish the balls which come down.  This is what was done in Butler, Chung, Cummings and Graham \cite{BCCG}, an example of such a card is shown in Figure~\ref{fig:separate}.  This avoids ambiguity that might arise but does not accurately reflect multiplex juggling in practice, but rather reflects passing patterns with multiple jugglers involved each juggler catching one ball (in particular the different points that come down correspond to the different jugglers).

\begin{figure}[htb!]
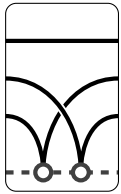

\centering
\picBCCGa
\caption{An example of a card used in Butler, Chung, Cummings and Graham \cite{BCCG}.}
\label{fig:separate}
\end{figure}

In this paper we will propose a new type of card which can be used for multiplex juggling.  It solves the ambiguity problem of Ehrenbrog and Readdy and also solves the modeling problem of Butler, Chung, Cummings and Graham.  However it does come at the mild cost of having a card being dependent on the \emph{previous} card which came before.  In Section~\ref{sec:cards} we will introduce these cards, and in Section~\ref{sec:graph} we show how to use matrices associated with a weighted graph to count the number of periodic patterns of length $n$.  We then count the number of siteswap sequences and the number of juggling patterns in Section~\ref{sec:properties} and Section~\ref{sec:counts}.  In Section~\ref{sec:hand} we will consider what happens when we limit the number of balls which can be thrown. We will give some concluding remarks in Section~\ref{sec:conclusion}, including a discussion of counting crossing numbers.

Most of the enumeration techniques here are fairly standard, it is their application to counting juggling patterns that is new.  We will also see that the objects generated in the process of deriving our count seem to have independent combinatorial interest.  Moreover, while our main goal has been to enumerate juggling patterns, the cards themselves might be useful for the exploration of other combinatorial aspects of juggling.

Finally, we note that while there has some been interest in counting multiplex juggling patterns, prior to this paper there has been little success.  Butler and Graham \cite{BG} made the most progress but their focus was on counting closed walks in a state graph and were not able to efficiently enumerate all juggling patterns.

\section{Cards for multiplex juggling}\label{sec:cards}
The way that cards describe juggling patterns is through understanding the relative order of their landing times.  The ambiguity that appeared in Figure~\ref{fig:nonunique} comes from the fact that two balls are landing \emph{together} but still being kept \emph{separate} in the ordering.  Since they are separate we could order them in two ways but that does not effect the pattern.  This suggests the following simple fix: tracks no longer represent individual balls, but rather groups of balls which will land together.   So now either the ``lowest'' group doesn't land, or the lowest group lands and the balls get thrown so that they are placed in new track(s) or added to the existing tracks.

Before each throw we will have an ordered partition of the number of balls $b$ on the left, i.e., $(q_1,q_2,\ldots,q_k)$ which corresponds to the statement that were we to stop juggling we would first have $q_1$ balls land at some point; then $q_2$ balls land some time later; and so on until finally $q_k$ balls land at the end.  (Note that we do not claim that they will land one right after the other; cards are keeping track of relative ordering of when things land and not the absolute times that they will land.)  Similarly after each throw we will have another ordered partition of $b$ on the right, i.e., $(r_1,r_2,\ldots,r_\ell)$.  (The number of our parts in our two partitions need not be the same but we must have $\ell\ge k-1$.)  If anything lands then the card in the middle indicates how the $q_1$ balls get redistributed.  Examples of these cards are shown in Figure~\ref{fig:multicards} where the first card corresponds to going from $(2,1,1)$ back to $(2,1,1)$ and the second corresponds to going from $(2,2,1)$ to $(1,2,2)$.

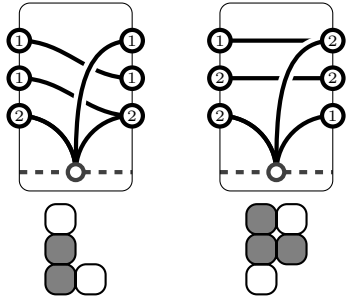
\begin{figure}[htb!]
\centering
\begin{tabular}{c@{\qquad}c}
\picJUGA & \picJUGB\\
\begin{tikzpicture}[scale=0.4]
\fill[color=black!50!white,rounded corners] (0,0) rectangle (1,1) (0,1) rectangle (1,2);
\draw[thick,rounded corners] (0,2) rectangle (1,3) (0,1) rectangle (1,2) (0,0) rectangle (1,1) (1,0) rectangle (2,1);
\end{tikzpicture}&
\begin{tikzpicture}[scale=0.4]
\fill[color=black!50!white,rounded corners] (0,1) rectangle (1,2) (1,1) rectangle (2,2) (0,2) rectangle (1,3);
\draw[thick, rounded corners] (0,0) rectangle (1,1) (0,1) rectangle (1,2) (1,1) rectangle (2,2) (0,2) rectangle (1,3) (1,2) rectangle (2,3);
\end{tikzpicture}
\end{tabular}
\caption{Cards which can be used to model multiplex juggling.}
\label{fig:multicards}
\end{figure}

\begin{definition}
An ordered partition $(q_1,q_2,\ldots,q_k)$ can be \emph{nontrivially embedded} into an ordered partition $(r_1,r_2,\ldots,r_\ell)$ if there exists indices $1\le i_2< i_3<\cdots< i_k\le\ell$ so that $q_j\le r_{i_j}$ for $2\le j\le k$.  Note that given two ordered partitions several nontrivial embeddings are possible.  An ordered partition $(q_1,q_2,\ldots,q_k)$ can be \emph{trivially embedded} into an ordered partition $(r_1,r_2,\ldots,r_\ell)$ if and only if  $(q_1,q_2,\ldots,q_k) = (r_1,r_2,\ldots,r_\ell)$.
\end{definition}

For every nontrivial embedding of $(q_1,q_2,\ldots,q_k)$, a partition of $b$, into $(r_1,r_2,\ldots,r_\ell)$, another partition of $b$, we have a card for multiplex juggling where a throw occurred.  As an example in Figure~\ref{fig:multicards} we have also marked underneath how $(q_1,q_2,\ldots,q_k)$ embeds into $(r_1,r_2,\ldots,r_\ell)$ by drawing the partition of $(r_1,r_2,\ldots,r_\ell)$ arranged from $r_1$ on the bottom to $r_\ell$ on the top and shading where $q_2,\ldots,q_k$ sits inside the partition.  Trivial embeddings, i.e.,  $(q_1,q_2,\ldots,q_k)=(r_1,r_2,\ldots,r_\ell)$,  correspond to no throws.  All possible cards (and corresponding embeddings of partition) for multiplex juggling when $b=3$ are shown in Figure~\ref{fig:3ballcards}.

\begin{figure}[h!]
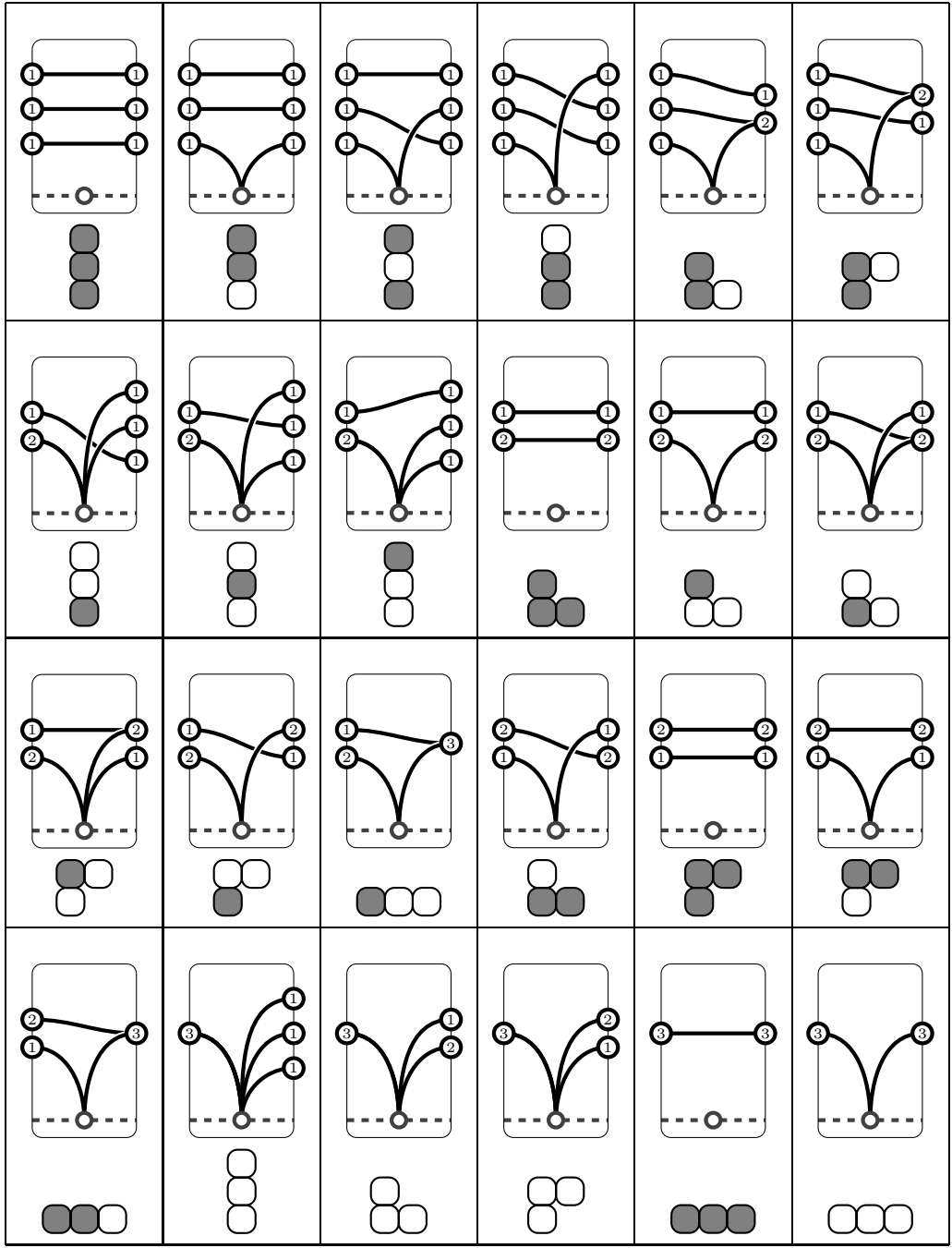

\centering
\begin{tabular}{|c|c|c|c|c|c|} 
\hline \vphantom{-}&&&&&\\
\picJPaa&\picJPab&\picJPac&\picJPad&\picJPae&\picJPaf\\
\partpicA&\partpicB&\partpicC&\partpicD&\partpicE&\partpicF\\ \hline \vphantom{-}&&&&&\\
\picEXg&\picEXf&\picEXe&\picJPbd&\picEXa&\picEXb\\
\partpicG&\partpicH&\partpicI&\partpicJ&\partpicK&\partpicL\\ \hline \vphantom{-}&&&&&\\
\picEXc&\picEXd&\picEXh&\picJPcd&\picJPce&\picJPcf\\
\partpicM&\partpicN&\partpicO&\partpicP&\partpicQ&\partpicR\\ \hline \vphantom{-}&&&&&\\
\picJPda&\picJPdb&\picJPdc&\picJPdd&\picJPde&\picJPdf\\
\partpicS&\partpicT&\partpicU&\partpicV&\partpicW&\partpicX\\ \hline
\end{tabular}
\caption{All cards and corresponding embeddings of ordered partitions for $b=3$.}
\label{fig:3ballcards}
\end{figure}

We can now determine the number of cards involved by examining their interpretation using embeddings of partitions.

\begin{lemma}\label{lem:right}
The total number of ways that an ordered partition $(r_1,r_2,\ldots,r_\ell)$ can have an ordered partition embedded inside is 
\[
\prod_i(r_i+1).
\]
Equivalently this is the number of cards where $(r_1,r_2,\ldots,r_\ell)$ is the ordered partition on the right.
\end{lemma}
\begin{proof}
We can think of shading the whole partition then removing parts of it as desired.  In particular for every part $r_i$ the embedding can use any of $0,1,\ldots,r_i$ in that slot.  In particular for the $i$th part there are $r_i+1$ choices and these can be made independently of each other.  Therefore there are $(r_1+1)(r_2+1)\cdots(r_\ell+1)$ possible embeddings.
\end{proof}

\begin{lemma}\label{lem:left}
The total number of ways that an ordered partition $(q_1,q_2,\ldots,q_k)$ can be embedded into an ordered partition is
\[
1+\frac{q_1+2k-2}{q_1+k-1}{q_1+k-1 \choose k-1}2^{q_1-1}.
\]
Equivalently this is the number of cards where $(q_1,q_2,\ldots,q_k)$ is the ordered partition on the left.
\end{lemma}
\begin{proof}
Suppose that a throw happens, so $q_1$ balls come down and get redistributed, possibly adding balls to existing groups to land (i.e., adding to one of the $q_i$) or creating new groups to land.  Suppose that the ordered partition we embed into has $(k-1)+\ell$ different parts; where $0\le \ell\le q_1$.  This can happen in 
\[
{k-1+\ell\choose \ell}{k+q_1-2\choose q_1-\ell}=\frac{\ell+k-1}{q_1+k-1}{q_1\choose \ell}{q_1+k-1\choose k-1}
\]
different ways.  The ${k-1+\ell\choose \ell}$ divides the parts as coming from an existing part (e.g., $q_i$) or being newly created.  For the $\ell$ new parts we first need to get a contribution of $1$ coming from $q_1$ leaving $q_1-\ell$ available to distribute among the $k-1+\ell$ different parts arbitrarily which can be done in  ${k+q_1-2\choose q_1-\ell}$ ways.  Finally, we can perform some simple manipulation of binomial coefficients.

Putting this together we have that the total number of ways that $(q_1,q_2,\ldots,q_k)$ can be embedded into another ordered partition by a throw is
\begin{align*}
\sum_{\ell=0}^{q_1}\frac{\ell+k-1}{q_1+k-1}{q_1\choose \ell}{q_1+k-1\choose k-1}
&=\frac{1}{q_1+k-1}{q_1+k-1\choose k-1}\sum_{\ell=0}^{q_1}(\ell+k-1){q_1\choose \ell}\\
&=\frac{1}{q_1+k-1}{q_1+k-1\choose k-1}\bigg(\sum_{\ell=0}^{q_1}\ell{q_1\choose \ell}
+\sum_{\ell=0}^{q_1}(k-1){q_1\choose \ell}\bigg)\\
&=\frac{1}{q_1+k-1}{q_1+k-1\choose k-1}\big(q_12^{q_1-1}+(k-1)2^{q_1}\big)\\
\\
&=\frac{q_1+2k-2}{q_1+k-1}{q_1+k-1 \choose k-1}2^{q_1-1}.
\end{align*}
Combining this with the ``$+1$'' from the trivial embedding and the result follows.
\end{proof}

We can now determine the total number of cards, or equivalently the total number of embeddings possible.  Starting with $b=0$ the numbers are
\[
1,
~2,
~7,
~24,
~82,
~280,
~956,
~3264,
~11144,
~38048,
~129904,
~443520,
~1514272,\ldots.
\]
This is sequence {\tt A003480} in the OEIS \cite{OEIS} which is initiated with $a_0=1$, $a_1=2$ and $a_2=7$ and for $b\ge 3$ we have $a_b=4a_{b-1}-2a_{b-2}$.  Verifying the first few cases is straightforward (the case for $a_3=24$ is shown in Figure~\ref{fig:3ballcards}).  It remains to verify the recurrence.

\begin{theorem}\label{thm:cardrecurse}
If $a_b$ is the number of possible cards used for describing multiplex juggling with $b$ balls, then for $b\ge 3$ we have $a_b=4a_{b-1}-2a_{b-2}$.
\end{theorem}
\begin{proof}
We can count the number of cards by breaking our count up according to the ordered partition that shows up on the right side using Lemma~\ref{lem:right}.  In particular we have
\begin{equation}\label{eq:ab}
a_b=\sum_{(r_1,r_2,\ldots)\in\mathcal{R}_b}\prod_{i\ge1}(r_i+1),
\end{equation}
where $\mathcal{R}_b$ are all of the ordered partitions of $b$.  

We now further break up this count by combining the ordered partitions according to the size of the first part which can be anything from $1$ to $b$.  So we have
\begin{align*}
a_b&=\sum_{j=1}^b\bigg(\sum_{(j,r_2,\ldots)\in\mathcal{R}_b}(j+1)\prod_{i\ge 2}(r_i+1)\bigg)\\
&=\sum_{j=1}^b(j+1)\bigg(\sum_{(r_2,\ldots)\in\mathcal{R}_{b-j}}\prod_{i\ge 2}(r_i+1)\bigg)\\
&=\sum_{j=1}^b(j+1)a_{b-i},
\end{align*}
where in the last step we note that we have the form of \eqref{eq:ab}.

To finish we have
\begin{align*}
a_b-2a_{b-1}+a_{b-2}&=
\sum_{j=1}^b(j+1)a_{b-j} 
-2\sum_{j=1}^{b-1}(j+1)a_{b-1-j}
+\sum_{j=1}^{b-2}(j+1)a_{b-2-j}\\
&=\sum_{k=1}^b(k+1)a_{b-k} 
-2\sum_{k=2}^{b}ka_{b-k}
+\sum_{k=3}^{b}(k-1)a_{b-k}\\
&=2a_{b-1}-a_{b-2}.
\end{align*}
Here we reindex the three sums to be consistent and then note that all but the first two terms will drop out.  Rearranging we conclude $a_b=4a_{b-1}-2a_{b-2}$, as desired.
\end{proof}

\section{Combining cards by taking walks in a graph}\label{sec:graph}

The advantage to using cards in order to describe juggling patterns was the ability to have our current card be chosen independently of all other cards.  With these new cards that we will use for multiplex juggling we now have to be careful in that the choice of our current card is dependent on the previous card, namely the partition on the right of the previous card must match the partition on the left of the current card.  Moreover if we are forming a pattern with period $n$ we need to make sure that our last card will also be consistent with our first card (so that we can repeat).

To help achieve this we will construct a directed multi-graph $\mathcal{G}_b$ by letting the vertices of $\mathcal{G}_b$ be the ordered partitions of $b$; and for each card we add a directed edge from the ordered partition on the left of the card to the ordered partition on the right of the card.

\begin{observation}\label{obs:bijection}
There is a one-to-one correspondence between periodic sequences using $n$ cards and closed walks of length $n$ in the graph.  In particular, to count the number of periodic sequences using $n$ cards it suffices to count the number of closed walks of length $n$.
\end{observation}

This follows by noting that each card is an edge and if two cards are used sequentially then the edges also occur sequentially in the graph, giving the correspondence between sequences of consecutive cards and walks in the graph.  Moreover the fact that we can repeat the pattern periodically indicates that we must return to the same ordered partition that we started with, so the walk is closed.

We now can use the transfer matrix method (see Stanley \cite[Ch.\ 4.7]{S}).

\begin{theorem}\label{thm:walks}
Given a directed multi-graph $\mathcal{G}$ let $A$ be the matrix with rows and columns indexed by the vertices with $A_{uv}$ equal to the number of directed arcs from $u\to v$.  Then $\trace(A^n)$ equals the number of closed walks of length $n$ in the graph.
\end{theorem}

Let $A_b$ be the matrix associated with the graph $\mathcal{G}_b$.  We have $A_0=(1)$, $A_1=(2)$,
\[
A_2=\begin{array}{r}(2)\\(1,1)\end{array}\left(\begin{array}{cc}
2 & 1 \\
1 & 3
\end{array}\right),\qquad
A_3=
\begin{array}{r}(3)\\(2,1)\\(1,2)\\(1,1,1)\end{array}\left(\begin{array}{cccc}
2 & 1 & 1 & 1 \\
1 & 3 & 2 & 3 \\
1 & 1 & 2 & 0 \\
0 & 1 & 1 & 4
\end{array}\right),
\]
and
\[
A_4=       	
\begin{array}{r}(4)\\(3,1)\\(1, 3)\\(2, 1, 1)\\(2, 2)\\(1, 2, 1)\\(1, 1, 2)\\(1, 1, 1, 1)\end{array}\left(\begin{array}{cccccccc}
2 & 1 & 1 & 1 & 1 & 1 & 1 & 1 \\
1 & 3 & 2 & 3 & 2 & 3 & 3 & 4 \\
1 & 1 & 2 & 0 & 0 & 0 & 0 & 0 \\
0 & 1 & 1 & 4 & 1 & 3 & 3 & 6 \\
1 & 1 & 1 & 1 & 3 & 1 & 1 & 0 \\
0 & 1 & 0 & 2 & 1 & 2 & 0 & 0 \\
0 & 0 & 1 & 0 & 1 & 1 & 3 & 0 \\
0 & 0 & 0 & 1 & 0 & 1 & 1 & 5

\end{array}\right),
\]
where on the left we have marked the ordered partitions that correspond to the vertex.  For reference we also give the graphs $\mathcal{G}_0$, $\mathcal{G}_1$, $\mathcal{G}_2$, $\mathcal{G}_3$ in Figure~\ref{fig:graphs}.

\begin{figure}[htb!]
\centering
\begin{tabular}{c@{\qquad\qquad}c@{\qquad\qquad}c@{\qquad\qquad}c}
\begin{tikzpicture}
\node[draw, rectangle, rounded corners, inner sep=2pt] (a) at (0,0) {$\emptyset$};
\draw[-{>[scale=1.25]},thick,shorten >=2pt,shorten <=2pt, rounded corners] (a) -- (0.75,0)--(0.75,0.75) node[circle,draw,thick, fill=white, inner sep = 1pt] {\small $1$} --(0,0.75)--(a);
\end{tikzpicture}

&

\begin{tikzpicture}
\node[draw, rectangle, rounded corners, inner sep=2pt] (a) at (0,0) {$(1)$};
\draw[-{>[scale=1.25]},thick,shorten >=2pt,shorten <=2pt, rounded corners] (a) -- (0.75,0)--(0.75,0.75) node[circle,draw,thick, fill=white, inner sep = 1pt] {\small $2$} --(0,0.75)--(a);
\end{tikzpicture}

&

\begin{tikzpicture}
\node[draw, rectangle, rounded corners, inner sep=2pt] (a) at (0,0) {$(1,1)$};
\node[draw, rectangle, rounded corners, inner sep=2pt] (b) at (0,2) {$(2)$};
\draw[-{>[scale=1.25]},thick,shorten >=2pt,shorten <=2pt, rounded corners] (a) -- (0.75,0)--(0.75,-0.75) node[circle,draw,thick, fill=white, inner sep = 1pt] {\small $3$} --(0,-0.75)--(a);
\draw[-{>[scale=1.25]},thick,shorten >=2pt,shorten <=2pt, rounded corners] (b) -- (0.75,2)--(0.75,2.75) node[circle,draw,thick, fill=white, inner sep = 1pt] {\small $2$} --(0,2.75)--(b);
\draw[-{>[scale=1.25]},thick,shorten >=2pt,shorten <=2pt, rounded corners] (b) to (0.25,1) node[circle,draw,thick, fill=white, inner sep = 1pt] {\small $1$} to (a);
\draw[-{>[scale=1.25]},thick,shorten >=2pt,shorten <=2pt, rounded corners] (a) to (-0.25,1) node[circle,draw,thick, fill=white, inner sep = 1pt] {\small $1$} to (b);
\end{tikzpicture}
&
\begin{tikzpicture}
\node[draw, rectangle, rounded corners, inner sep=2pt] (a) at (0,0) {$(3)$};
\node[draw, rectangle, rounded corners, inner sep=2pt] (b) at (3,0) {$(1,2)$};
\node[draw, rectangle, rounded corners, inner sep=2pt] (c) at (3,3) {$(1,1,1)$};
\node[draw, rectangle, rounded corners, inner sep=2pt] (d) at (0,3) {$(2,1)$};
\draw[-{>[scale=1.25]},thick,shorten >=2pt,shorten <=2pt, rounded corners] (a) -- (-0.75,0)--(-0.75,-0.75) node[circle,draw,thick, fill=white, inner sep = 1pt] {\small $2$} --(0,-0.75)--(a);
\draw[-{>[scale=1.25]},thick,shorten >=2pt,shorten <=2pt, rounded corners] (d) -- (-0.75,3)--(-0.75,3.75) node[circle,draw,thick, fill=white, inner sep = 1pt] {\small $3$} --(0,3.75)--(d);
\draw[-{>[scale=1.25]},thick,shorten >=2pt,shorten <=2pt, rounded corners] (b) -- (3.75,0)--(3.75,-0.75) node[circle,draw,thick, fill=white, inner sep = 1pt] {\small $2$} --(3,-0.75)--(b);
\draw[-{>[scale=1.25]},thick,shorten >=2pt,shorten <=2pt, rounded corners] (c) -- (4,3)--(4,3.75) node[circle,draw,thick, fill=white, inner sep = 1pt] {\small $4$} --(3,3.75)--(c);
\draw[-{>[scale=1.25]},thick,shorten >=2pt,shorten <=2pt, rounded corners] 
    (a) to (1.5,0.3) node[circle,draw,thick, fill=white, inner sep = 1pt] {\small $1$} to (b);
\draw[-{>[scale=1.25]},thick,shorten >=2pt,shorten <=2pt, rounded corners] 
    (b) to (1.5,-0.3) node[circle,draw,thick, fill=white, inner sep = 1pt] {\small $1$} to (a);
\draw[-{>[scale=1.25]},thick,shorten >=2pt,shorten <=2pt, rounded corners] 
    (a) to (-0.3,1.5) node[circle,draw,thick, fill=white, inner sep = 1pt] {\small $1$} to (d);
\draw[-{>[scale=1.25]},thick,shorten >=2pt,shorten <=2pt, rounded corners] 
    (d) to (0.3,1.5) node[circle,draw,thick, fill=white, inner sep = 1pt] {\small $1$} to (a);
\draw[-{>[scale=1.25]},thick,shorten >=2pt,shorten <=2pt, rounded corners] 
    (a) to (1.5,1.5) node[circle,draw,thick, fill=white, inner sep = 1pt] {\small $1$} to (c);
\draw[-{>[scale=1.25]},thick,shorten >=2pt,shorten <=2pt, rounded corners] 
    (c) to (3,1.5) node[circle,draw,thick, fill=white, inner sep = 1pt] {\small $1$} to (b);
\draw[-{>[scale=1.25]},thick,shorten >=2pt,shorten <=2pt, rounded corners] 
    (d) to (1.5,3.3) node[circle,draw,thick, fill=white, inner sep = 1pt] {\small $3$} to (c);
\draw[-{>[scale=1.25]},thick,shorten >=2pt,shorten <=2pt, rounded corners] 
    (c) to (1.5,2.7) node[circle,draw,thick, fill=white, inner sep = 1pt] {\small $1$} to (d);
\draw[-{>[scale=1.25]},thick,shorten >=2pt,shorten <=2pt, rounded corners] 
    (d) to (1.5,2.1) node[circle,draw,thick, fill=white, inner sep = 1pt] {\small $2$} to (b);
\draw[-{>[scale=1.25]},thick,shorten >=2pt,shorten <=2pt, rounded corners] 
    (b) to (1.5,0.9) node[circle,draw,thick, fill=white, inner sep = 1pt] {\small $1$} to (d);
\end{tikzpicture}
\\
$\mathcal{G}_0$& $\mathcal{G}_1$& $\mathcal{G}_2$& $\mathcal{G}_3$
\end{tabular}
\caption{The graphs $\mathcal{G}_0$, $\mathcal{G}_1$, $\mathcal{G}_2$ and $\mathcal{G}_3$.}
\label{fig:graphs}
\end{figure}
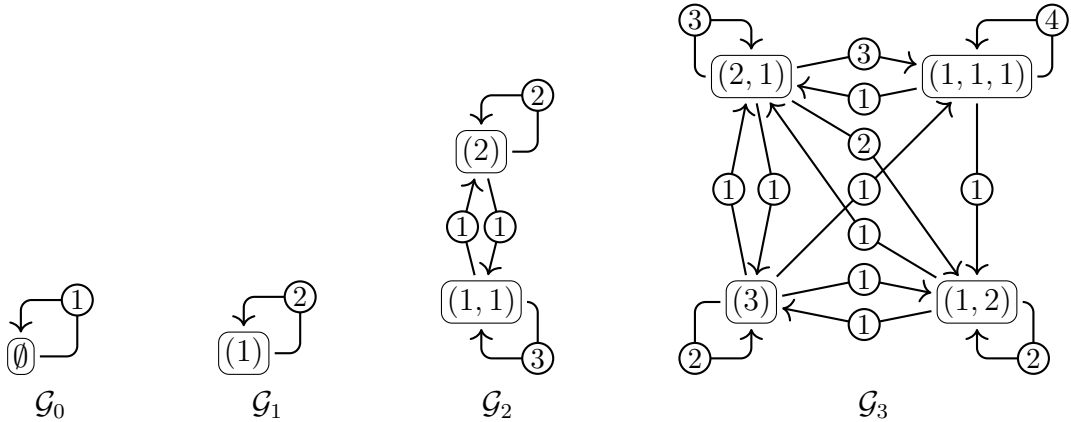

We note that Lemma~\ref{lem:right} can be used to determine the \emph{column} sums of $A_b$, while Lemma~\ref{lem:left} can be used to determine the \emph{row} sums of $A_b$.  Using Theorem~\ref{thm:walks} we also get the sum of all entries of the matrix.

\section{Counting multiplex siteswaps}\label{sec:properties}

We can combine Observation~\ref{obs:bijection} together with Theorem~\ref{thm:walks} to find the number of periodic sequences of $n$ cards.  For each sequence of cards we will get a siteswap pattern (recall that siteswap patterns work by recording how many beats in the future the ball(s) land which can be done by following the path along the figure formed by the cards until it goes back down).  However for a given siteswap there can be \emph{multiple} ways in which it can be represented using the cards.

The problem lies in what happens with balls that are never used.  For instance in Figure~\ref{fig:samesiteswap} we see two distinct sets of three cards which correspond to the same siteswap, the difference between them being the tracks in the unused balls.  (In siteswap notation this will be 1[112][22].)

\begin{figure}[htb!]
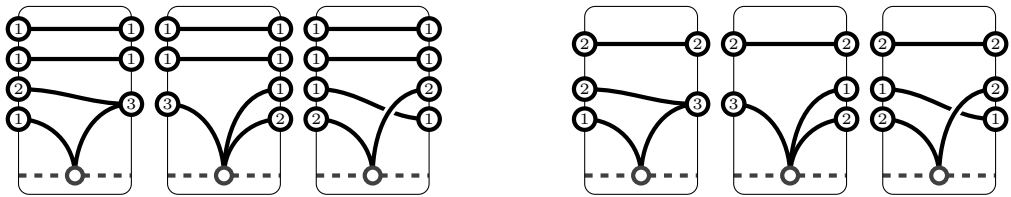

\centering
\picWeda \picWedb \picWedc
\hfil
\picWedd \picWede \picWedf
\caption{Two distinct sets of three cards with the same siteswap (1[112][22]).}
\label{fig:samesiteswap}
\end{figure}

Note that for normal juggling the unused balls all have the same behavior, in our language they would correspond to lying in tracks with capacity $1$.  So to count the number of siteswaps we take the difference of the ways to place $n$ cards with $b$ tracks $(b+1)^n$ and the ways to place $n$ cards with $b-1$ tracks $b^n$ (which is in bijection with the number of ways to place $n$ cards with $b$ tracks and the top track is never used).  This gives $(b+1)^n-b^n$.  We will do something similar with our cards that we are using for multiplex juggling.

\begin{theorem}\label{thm:siteswapcount}
Let $ss_b(n)$ be the number of siteswap patterns using exactly $b$ balls and with period $n$.  Then
\[
ss_b(n) = \trace(A_b^n)-\sum_{i=0}^{b-1}\trace(A_i^n).
\]
\end{theorem}
\begin{proof}
We proceed by induction on $b$.  For the base case of $b=0$ there is one card (the empty card) and so for any length $n$ there is exactly one way to position these cards.  (In siteswap notation this is 00\ldots0.)  Since $A_0=(1)$ we have $\trace(A_0^n)=1$ establishing the base case.

Now assume that it works up through $b-1$ and consider the case for $b$.  First we observe
\[
\trace(A_{b}^n) = ss_{b}(n)+\sum_{i=0}^{b-1}2^{b-i-1}ss_i(n).
\]
This follows by noting that for every siteswap which uses exactly $i$ balls we can find a sequence of cards corresponding to the ordered partitions of $i$.  We can now add $b-i$ balls to the top of each card as long as we do it consistently across the different cards.  Moreover the number of different options we have to place the $b-i$ balls equals the number of ordered partitions of $b-i$ which is $2^{b-i-1}$.

Rearranging and using our induction hypothesis we have
\begin{align*}
ss_b(n)
&=\trace(A_{b}^n)-\sum_{i=0}^{b-1}2^{b-i-1}ss_i(n)\\
&=\trace(A_{b}^n)-\sum_{i=0}^{b-1}2^{b-i-1}\bigg(\trace(A_i^n)-\sum_{j=0}^{i-1}\trace(A_j^n)\bigg)\\
&=\trace(A_{b}^n)-\sum_{i=0}^{b-1}\trace(A_i^n)\bigg(2^{b-i-1}-\sum_{k=0}^{b-i-2}2^k\bigg)\\
&=\trace(A_{b}^n)-\sum_{i=0}^{b-1}\trace(A_i^n).
\end{align*}
Establishing the result.
\end{proof}

\section{Counting multiplex juggling patterns}\label{sec:counts}
To go from counting siteswap patterns of period $n$ to counting juggling patterns of minimal period $n$ we want to do two things.  First we want to remove any pattern that has shorter period (suppose $d$ divides $n$, then any periodic sequence of cards of length $d$ can be repeated $n/d$ times to make a periodic sequence of cards of length $n$).  Second we want to divide out by $n$ since in juggling patterns there is no set start point, i.e., 1[112][22], [112][22]1, and [22]1[112] all correspond to the same juggling pattern.

For the first issue we can use M\"obius inversion (see \cite{S}).  Namely we note that if $ms_b(n)$ is the number of siteswap patterns with $b$ balls and minimal period $n$, then
\[
ss_b(n) = \sum_{d \mid n}ms_b(d).
\]
So if we let $\mu(\frac{n}{d})$ be the M\"obius function it follows
\[
ms_b(n) = \sum_{d \mid n}\mu({\textstyle\frac{n}{d}})ss_b(d).
\]
With this in hand we can now divide out by the rotational symmetry of the starting point and determine the number of juggling patterns with $b$ balls and period $n$ which we denote $jp_b(n)$.  Combining the above with Theorem~\ref{thm:siteswapcount} we get the following.

\begin{theorem}\label{thm:jpcount}
The number of juggling patterns with $b$ balls and minimal period $n$ is
\begin{equation*}
jp_b(n)=\frac1n\sum_{d\mid n}\mu({\textstyle\frac{n}{d}})\bigg(\trace(A_b^d)-\sum_{i=0}^{b-1}\trace(A_i^d)\bigg).
\end{equation*}
\end{theorem}

In Table~\ref{tab:jpcount} we give the number of minimal period juggling patterns for $b=2,3,4,5$ and period at most $15$.

\begin{table}[htb!]
\centering
\begin{tabular}{|l||r|r|r|r|}\hline
& $b=2$ & $b=3$ & $b=4$ & $b=5$\\ \hline\hline
$n=1$& $2$& $3$& $5$& $7$\\ \hline
$n=2$& $4$& $12$& $32$& $77$\\ \hline
$n=3$& $13$& $63$& $261$& $964$\\ \hline
$n=4$& $37$& $310$& $2089$& $12086$\\ \hline
$n=5$& $118$& $1618$& $17449$& $156975$\\ \hline
$n=6$& $356$& $8434$& $147807$& $2077448$\\ \hline
$n=7$& $1142$& $45142$& $1276577$& $27976399$\\ \hline
$n=8$& $3620$& $243998$& $11169023$& $381752857$\\ \hline
$n=9$& $11744$& $1336644$& $98872035$& $5267354817$\\ \hline
$n=10$& $38275$& $7392117$& $883717142$& $73358245986$\\ \hline
$n=11$& $126234$& $41247234$& $7964898829$& $1029873201879$\\ \hline
$n=12$& $418735$& $231856131$& $72305691686$& $14559160765380$\\ \hline
$n=13$& $1399610$& $1311820110$& $660528998007$& $207076019661773$\\ \hline
$n=14$& $4702499$& $7464002451$& $6067348742573$& $2961063646029819$\\ \hline
$n=15$& $15883190$& $42679372930$& $56002661734041$& $42542385162393167$\\ \hline
\end{tabular}
\caption{The number of multiplex juggling patterns of minimal period $n$ using $b$ balls.}
\label{tab:jpcount}
\end{table}

As a special case of Theorem~\ref{thm:jpcount} we have
\begin{equation}\label{eq:jp1}
jp_b(1)=\trace(A_b)-\sum_{i=0}^{b-1}\trace(A_i).
\end{equation}
We can also compute the number of period one multiplex juggling patterns directly.
\begin{theorem}\label{thm:jp1}
We have $jp_b(1)=p(b)$, where $p(b)$ is the number of unordered partitions of $b$.
\end{theorem}
\begin{proof}
In order for a card to produce a valid juggling sequence of period one with $b$ balls we must have that the (ordered) partitions on the left and right side of the card be equal.  That is $(q_1,q_2,\ldots,q_k)=(r_1,r_2,\ldots,r_k)$.  Further we must have that when the $q_1$ balls get distributed some ball gets placed into the top group.  This second requirement will force $q_k$ to ``embed'' into $r_{k-1}$ (otherwise it would have to embed into $r_k$ but the placement of at least one more ball in the top group then results in $q_k\ne r_k$).  In particular this then forces, in turn, $q_i$ to embed in $r_{i-1}$ for $i=2,\ldots,k$.  This is only possible if $q_1\le q_2\le\cdots\le q_k$.  Therefore our ordered partition on the sides of the card have the unique ordering from largest to smallest element.

Conversely, if we start with an unordered partition we can create a card as above by placing the partition on the sides of the card from largest to smallest; then all but the bottom group will shift down by one while the bottom group will then drop down and redistribute to fill differences as needed.  An example of this is shown in Figure~\ref{fig:partition} for the partition $(3,3,2,2).$
\end{proof}

\begin{figure}[htb!]
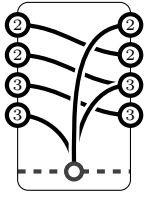

\centering
\picThu
\caption{Forming a period one juggling sequence from the partition $(3,3,2,2)$.}
\label{fig:partition}
\end{figure}

\begin{theorem}\label{thm:trace}
Let $p(b)$ denote the number of unordered partitions of $b$.  Then
\[
\trace(A_b)=p(b)+\sum_{i=0}^{b-1}2^{b-i-1}p(i).
\]
\end{theorem}
\begin{proof}
Updating equation \eqref{eq:jp1} using Theorem~\ref{thm:jp1} we have
\[
p(b) = \trace(A_b)-\sum_{i=0}^{b-1}\trace(A_i).
\]
We now proceed by induction.  We have $\trace(A_0)=1=p(0)$ (note that the sum will be empty and not contribute), establishing the base case.  Now suppose that we have established the result up through $b-1$ and consider the case for $b$.
\begin{align*}
\trace({A}_b)&=p(b)+\sum_{i=0}^{b-1}\trace(A_i)\\
&=p(b)+p(b-1)+2\sum_{i=0}^{b-2}\trace(A_i)\\
&=p(b)+p(b-1)+2p(b-2)+2^2\sum_{i=0}^{b-3}\trace(A_i)\\
&=\cdots \\
&=p(b)+p(b-1)+2p(b-2)+2^2p(b-3)+\cdots+2^{b-1}p(0)\\
&=p(b)+\sum_{i=0}^{b-1}2^{b-i-1}p(i)
\end{align*}
Establishing the result.
\end{proof}

Using this we get that the trace of $A_b$ starting with $b=0$ is 
\[
1, ~2, ~5, ~11, ~24, ~50, ~104, ~212, ~431, ~870, ~1752, ~3518, ~7057, ~14138, ~28310, ~56661, \ldots.
\]
This is sequence {\tt A090764} in the OEIS \cite{OEIS}.  The recurrence that we have derived is a variant of that given in the OEIS.  We now give a different variation establishing that this is the same sequence as given in the OEIS.

\begin{theorem}
Let $\mathcal{P}(b)$ be the set of unordered partitions of $b$, and for a partition $q$ let $o(q)$ denote the number of $1$s in the partition.  Then
\[ 
\trace({A}_b)=\sum_{q\in \mathcal{P}(b)}2^{o(q)}.
\]
\end{theorem}

\begin{proof}
We proceed by induction on $b$.  The result is easily checked for small cases, so now assume it works up through $b-1$.  Then we have the following.
\begin{align*}
\sum_{q \in \mathcal{P}_b}2^{o(q)}&=
\sum_{\substack{q\in\mathcal{P}_b\\ o(q)=0}} 2^{o(q)}+\sum_{\substack{q\in\mathcal{P}_b \\ o(q)>0}}2^{o(q)}\\
&=\sum_{\substack{q\in\mathcal{P}_b\\o(q)=0}} 1+2\sum_{q\in\mathcal{P}_{b-1}}2^{o(q)}\\
&=\big(p(b)-p(b-1)\big)+2\trace(A_{b-1})\\
&=\big(p(b)-p(b-1)\big)+2 \bigg(p(b-1)+\sum_{i=0}^{b-2}2^{b-i-2}p(i)\bigg)\\
&=p(b)+\sum_{i=0}^{b-1}2^{b-i-1}p(i)\\
&=\trace(A_{b}).
\end{align*}
We first divide up our partitions on whether we include a $1$; then we note that there are $p(b)-p(b-1)$ partitions that do not include a $1$ and $p(b-1)$ partitions that do include a $1$ (i.e., taking a partition with $b-1$ we can append a $1$ to get a partition of $b$ which does include a $1$).  We apply the induction hypothesis, and then apply Theorem~\ref{thm:trace} once, clean up our sum, and finally apply Theorem~\ref{thm:trace} a second time.
\end{proof}

\section{Juggling patterns with hand capacity}\label{sec:hand}
One of our basic assumptions that we have employed is that we can catch and throw any number of balls at any given step.  From a practical standpoint jugglers usually limit themselves to catching and throwing at most two or three balls at a time.  The approach outlined above works just as well when we introduce a capacity constraint into how many balls can land at a given time, or equivalently how many balls can be thrown at a given time.

To do this we let $A_{b,\kappa}$ be the principal submatrix of $A_b$ by taking all rows and columns indexed by ordered partitions with all parts of size at most $\kappa$.

\begin{theorem}\label{thm:capacity}
Let $ss_{b,\kappa}(n)$ be the number of siteswap patterns using exactly $b$ balls, with period $n$ and all throws involve at most $\kappa$ balls.  Then
\begin{equation}\label{eq:capss}
ss_{b,\kappa}(n)=\trace(A_{b,\kappa}^n)-\sum_{i=\max\{0,b-\kappa\}}^{b-1}\trace(A_{i,\kappa}^n).
\end{equation}
Let $jp_{b,\kappa}(n)$ be the number of juggling patterns using exactly $b$ balls, with minimal period $n$ and all throws involve at most $\kappa$ balls.  Then
\begin{equation}\label{eq:capjp}
jp_{b,\kappa}(n)=\frac1n\sum_{d\mid n}\mu({\textstyle\frac{n}{d}})\bigg(\trace(A_{b,\kappa}^d)-\sum_{i=\max\{0,b-\kappa\}}^{b-1}\trace(A_{i,\kappa}^d)\bigg).
\end{equation}
\end{theorem}

Before beginning the proof, observe if $\kappa=1$, then this reduces to the case of normal juggling (i.e., $A_{b,1}=(b+1)$); and if $\kappa=\infty$ then this is equivalent to what we have already done.

\begin{proof}
We note that \eqref{eq:capjp} follows from \eqref{eq:capss} by applying M\"obius inversion.  So it suffices to establish \eqref{eq:capss}.

We proceed by induction on $b$. For the base case of $b=0$ there is one card (the empty card) and the capacity places no restriction on its use, and so for any length there is exactly one way to position these cards.  Since $\trace(A_0^n)=1$ this establishes the base case.

Let $r_{i,\kappa}$ be the number of ordered partitions of $i$ with each part at most $\kappa$.  By grouping based on the first part (which has size between $1$ and $\kappa$) we have
\begin{equation}\label{eq:rrecur}
r_{i,\kappa}=\sum_{j=1}^{\min\{\kappa,i\}} r_{i-j,\kappa},
\end{equation}
the $\min\{\kappa,i\}$ coming from noting that we cannot have ordered partitions with negative parts and so we need to handle the case of small $i$.

Now assume that we have established the result up through $b-1$ and consider the case for $b$.  First we observe
\[
\trace(A_{b,\kappa}^n)=ss_{b,\kappa}(n)+\sum_{i=0}^{b-1}r_{b-i,\kappa}ss_{i,\kappa}(n).
\]
This follows by noting that for every siteswap which uses exactly $i$ balls we can find a sequence of cards which uses exactly  $i$ balls.  We can now add $b-i$ balls to the top of each card as long as we do it consistently across the different cards.  Moreover the number of different options we have to place the $b-i$ balls equals the number of ordered partitions of $b-i$ with each part at most $\kappa$ which is $r_{b-i,\kappa}$.

Rearranging and using our induction hypothesis we have
\begin{align*}
ss_{b,\kappa}(n)&=
\trace(A_{b,\kappa}^n)-\sum_{i=0}^{b-1}r_{b-i,\kappa}ss_{i,\kappa}(n)\\
&=\trace(A_{b,\kappa}^n)-\sum_{i=0}^{b-1}r_{b-i,\kappa}\bigg(\trace(A_{i,\kappa}^n)-\sum_{j=\max\{0,i-\kappa\}}^{i-1}\trace(A_{j,\kappa}^n)\bigg)\\
&=\trace(A_{b,\kappa}^n)-\sum_{i=0}^{b-1}\trace(A_{i,\kappa}^n)\bigg(r_{b-i,\kappa}-\sum_{j=\max\{0,b-i-\kappa\}}^{b-i-1}r_{j,\kappa}\bigg)\\
&=\trace(A_{b,\kappa}^n)-\sum_{i=\max\{0,b-\kappa\}}^{b-1}\trace(A_{i,\kappa}^n).
\end{align*}
In going from the first to the second line we use the induction hypothesis, and in going from the second to the third line we rearrange the terms.  Finally we observe that \eqref{eq:rrecur} indicates that almost all terms will $0$ out, except for the first few initial terms which can easily be checked to be one, establishing the result.
\end{proof}

For reference we have produced the number of juggling patterns for $\kappa=2$ in Table~\ref{tab:cap2} and for $\kappa=3$ in Table~\ref{tab:cap3}.

\begin{table}[htb!]
\centering
\begin{tabular}{|l|r|r|r|r|}\hline
$\kappa=2$ & $b=2$ & $b=3$ & $b=4$ & $b=5$\\ \hline\hline
$n=1$& $2$& $2$& $3$& $3$\\ \hline
$n=2$& $4$& $9$& $18$& $30$\\ \hline
$n=3$& $13$& $47$& $134$& $314$\\ \hline
$n=4$& $37$& $224$& $950$& $3140$\\ \hline
$n=5$& $118$& $1118$& $6938$& $31886$\\ \hline
$n=6$& $356$& $5522$& $50751$& $324909$\\ \hline
$n=7$& $1142$& $27910$& $376402$& $3341566$\\ \hline
$n=8$& $3620$& $141946$& $2813824$& $34605634$\\ \hline
$n=9$& $11744$& $730544$& $21219536$& $360849352$\\
\hline
$n=10$& $38275$& $3790391$& $161190485$& $3785776259$\\
\hline
$n=11$& $126234$& $19827570$& $1232724798$&
$39941119938$\\ \hline
$n=12$& $418735$& $104422007$& $9483975303$&
$423549648963$\\ \hline
$n=13$& $1399610$& $553339258$& $73360425430$&
$4512516867634$\\ \hline
$n=14$& $4702499$& $2947940371$& $570219618745$&
$48282551418859$\\ \hline
$n=15$& $15883190$& $15780565950$& $4451677886746$&
$518633980103198$\\ \hline
\end{tabular}
\caption{The number of juggling patterns of period $n$ using $b$ balls with capacity $2$.}
\label{tab:cap2}
\end{table}

\begin{table}[htb!]
\centering
\begin{tabular}{|l|r|r|r|r|}\hline
$\kappa=3$ & $b=2$ & $b=3$ & $b=4$ & $b=5$\\ \hline\hline	
$n=1$& $2$& $3$& $4$& $5$\\ \hline
$n=2$& $4$& $12$& $28$& $58$\\ \hline
$n=3$& $13$& $63$& $231$& $713$\\ \hline
$n=4$& $37$& $310$& $1840$& $8591$\\ \hline
$n=5$& $118$& $1618$& $15168$& $106073$\\ \hline
$n=6$& $356$& $8434$& $126258$& $1325570$\\ \hline
$n=7$& $1142$& $45142$& $1069002$& $16789985$\\ \hline
$n=8$& $3620$& $243998$& $9154845$& $214916096$\\ \hline
$n=9$& $11744$& $1336644$& $79252442$& $2776778019$\\
\hline
$n=10$& $38275$& $7392117$& $692290928$& $36167946945$\\
\hline
$n=11$& $126234$& $41247234$& $6095630354$&
$474470288650$\\ \hline
$n=12$& $418735$& $231856131$& $54045188641$&
$6263882726811$\\ \hline
$n=13$& $1399610$& $1311820110$& $482108239540$&
$83162406390939$\\ \hline
$n=14$& $4702499$& $7464002451$& $4323812672665$&
$1109678347266127$\\ \hline
$n=15$& $15883190$& $42679372930$& $38963338572980$&
$14873888879020290$\\ \hline
\end{tabular}
\caption{The number of juggling patterns of period $n$ using $b$ balls with capacity $3$.}
\label{tab:cap3}
\end{table}

We note that we can ask similar questions about the sum of the entries in $A_{b,\kappa}$ as well as the row and column sums as we did with $A_{b}$ (i.e., Lemma~\ref{lem:right}, Lemma~\ref{lem:left} and Theorem~\ref{thm:cardrecurse}).  However the counts are less clear, and have not appeared in the OEIS.  As an example if we count the total number of cards when $\kappa =2$ we get the following numbers, starting with $b=0$:
\[
1,~2,~7,~17,~41,~91,~195,~403,~812,~1601,~3102,~5922,~11165,~20824,~38477,\ldots.
\]
These numbers do appear to satisfy a relatively simple relationship.  In particular through $b=25$ these numbers agree with the following conjecture.

\begin{conjecture}
Let $a_b^{(2)}$ be the number of cards for multiplex juggling with $b$ balls and where each track has capacity at most $2$.  Then
\[
\sum_{b\ge 0}a_b^{(2)}x^b = \frac{1-x+x^2+x^3}{(1-x-x^2)^3}.
\]
\end{conjecture}

\section{Conclusion}\label{sec:conclusion}
By modifying the cards used for juggling, namely allowing groups of balls to be grouped together, we have found a method that works for enumerating multiplex juggling patterns.  There are still a few questions that remain, particularly in understanding what happens when we limit the number of balls that can be caught at any given time.

\bigskip

Ehrenborg and Readdy \cite{ER} introduced cards for juggling and used them to study a $q$-analog of juggling by counting crossings.  It is easy to count crossings on each card and then one simply adds up the crossings over all cards used to count the crossings of the pattern.  We note that the matrices used here can be easily adapted to this situation.  Namely for each card we count crossings (making sure to count multiplicity when balls move in groups), and then weight the card (and hence edge in the graph) by $q^k$ where $k$ is the number of crossings.  Finally we can form matrices $A_b(q)$ where we add up the weights of cards connecting ordered partitions.  As an example we have
\[
A_3(q)=
\begin{array}{r}(3)\\(2,1)\\(1,2)\\(1,1,1)\end{array}\left(\begin{array}{cccc}
2 & 1 & 1 & 1 \\
1 & q+2 & q+1 & q^2+q+1 \\
1 & q & 2 & 0 \\
0 & 1 & q & q^2+q+2
\end{array}\right).
\]
We note that Theorem~\ref{thm:siteswapcount} and Theorem~\ref{thm:capacity} can be easily modified to count the number of juggling patterns with minimal period based on the number of crossings.

\bigskip

An applied mathematical juggler might also want to add the constraint that whenever multiple balls are thrown that no two balls get thrown to the same height.  Our method can be readily adopted to this situation by simply removing any card which has two balls moved to the same track, which leads to modified graphs $\widehat{\mathcal{G}}_b$, and also modified matrices, $\widehat{A}_b$.  For example we have
\[
\widehat{A}_2=\begin{array}{r}(2)\\(1,1)\end{array}\left(\begin{array}{cc}
1 & 1 \\
1 & 3
\end{array}\right),\qquad\text{and}\quad
\widehat{A}_3=
\begin{array}{r}(3)\\(2,1)\\(1,2)\\(1,1,1)\end{array}\left(\begin{array}{cccc}
1 & 0 & 0 & 1 \\
0 & 2 & 1 & 3 \\
1 & 1 & 2 & 0 \\
0 & 1 & 1 & 4
\end{array}\right).
\]

\bigskip

The matrices $A_b$ might also have independent interest.  For example, it is easy to see that for $1\le i\le 3$ that $A_i$ is a principal submatrix of $A_{i+1}$.  This seems to continue for at least the first few cases. 
Does this containment continue?  Note that this also seems to indicate a preferred ordering of the ordered partitions if we want to have (1) containment of the previous matrix in the upper left block and (2) an upper triangular matrix in the lower left block.  What is this ordering?

Things get even more interesting when we consider the characteristic polynomial of $A_b$.  If we let $P_b:=P_b(x)=\det(xI-A_b)$, then we have the following.
\[
\begin{array}{l@{\,=\,}llllllllllllll}
P_0&f_0^1\\
P_1&&f_1^1\\
P_2&&&f_2^1\\
P_3&f_0^1&&&f_3^1\\
P_4&f_0^2&f_1^1&&&f_4^1\\
P_5&f_0^5&f_1^2&f_2^1&&&f_5^1\\
P_6&f_0^9&f_1^5&f_2^2&f_3^1&&&f_6^1\\
P_7&f_0^{19}&f_1^9&f_2^5&f_3^2&f_4^1&&&f_7^1\\
P_8&f_0^{37}&f_1^{19}&f_2^9&f_3^5&f_4^2&f_5^1&&&f_8^1\\
P_9&f_0^{74}&f_2^{37}&f_2^{19}&f_3^9&f_4^5&f_5^2&f_6^1&&&f_9^1\\
P_{10}&f_0^{148}&f_1^{74}&f_2^{37}&f_3^{19}&f_4^9&f_5^5&f_6^2&f_7^1&&&f_{10}^1\\
P_{11}&f_0^{296}&f_1^{148}&f_2^{74}&f_3^{37}&f_4^{19}&f_5^9&f_6^5&f_7^2&f_8^1&&&f_{11}^1\\
P_{12}&f_0^{591}&f_1^{296}&f_2^{148}&f_3^{74}&f_4^{37}&f_5^{19}&f_6^9&f_7^5&f_8^2&f_9^1&&&f_{12}^1\\
P_{13}&f_0^{1183}&f_1^{591}&f_2^{296}&f_3^{148}&f_4^{74}&f_5^{37}&f_6^{19}&f_7^9&f_8^5&f_9^2&f_{10}^1&&&f_{13}^1
\end{array}
\]
Where the polynomials $f_i$ are given by the following.
\begin{align*}
f_0(x)&=x-1\\
f_1(x)&=x-2\\
f_2(x)&=x^2-5x+5\\
f_3(x)&=x^3 - 10x^2 + 27x -20\\
f_4(x)&=x^5 - 20x^4 + 135x^3 - 396x^2 + 518x - 245\\
f_5(x)&=x^7 - 36x^6 + \cdots - 2^5{\cdot}5{\cdot}7^2
\\
f_6(x)&=x^{11} - 65x^{10} + \cdots - 2^5 {\cdot} 3^2 {\cdot} 5^2 {\cdot} 7^3\\
f_7(x)&=x^{15} - 110x^{14} + \cdots - 2^{11} {\cdot} 3^2 {\cdot} 5^3 {\cdot} 7^4\\
f_8(x)&=x^{22} - 185x^{21} + \cdots + 2^{10} {\cdot} 3^4 {\cdot} 5^4 {\cdot} 7^6 {\cdot} 11^2\\
f_9(x)&=x^{30} - 300x^{29} + \cdots + 2^{21} {\cdot} 3^5 {\cdot} 5^6 {\cdot} 7^8 {\cdot} 11^2\\
f_{10}(x)&=x^{42} - 481x^{41} + \cdots + 2^{21} {\cdot} 3^9 {\cdot} 5^8 {\cdot} 7^{12} {\cdot} 11^3 {\cdot} 13^2
\\
f_{11}(x)&=x^{56} - 752x^{55} +\cdots + 2^{38} {\cdot} 3^{12} {\cdot} 5^{11} {\cdot} 7^{16} {\cdot} 11^4 {\cdot} 13^2\\
f_{12}(x)&=x^{77} - 1165x^{76}+\cdots - 2^{42} {\cdot} 3^{17} {\cdot} 5^{16} {\cdot} 7^{22} {\cdot} 11^9 {\cdot} 13^3\\
f_{13}(x)&=x^{101} - 1770x^{100} +\cdots - 2^{70} {\cdot} 3^{23} {\cdot} 5^{21} {\cdot} 7^{29} {\cdot} 11^{11} {\cdot} 13^4
\end{align*}

There are several interesting things that are appearing.
\begin{itemize}
\item The sequence of the exponents of $f_i$ in the $P_j$ seem to follow 
\[
1,~0,~0,~1,~2,~5,~9,~19,~37,~74,~148,~296,~591,~1183,\ldots.
\]
This appears to be the sequence {\tt A178841} in the OEIS \cite{OEIS} which counts the number of pure inverting compositions of $n$.
\item The degree of the polynomials $f_i$ follow
\[
1,~1,~2,~3,~5,~7,~11,~15,~22,~30,~42,~56,~77,~101,\ldots.
\]
This appears to be the sequence {\tt A000041} in the OEIS \cite{OEIS} which counts the number of unordered partitions of $n$.
\item The second coefficients of the polynomials $f_i$ follow
\[
1,~2,~5,~10,~20,~36,~65,~110,~185,~300,~481,~752,~1165,~1770,\ldots.
\]
This appears to be the sequence {\tt A000712} in the OEIS \cite{OEIS} which counts the number of unordered partitions of $n$ into parts of $2$ kinds.
\end{itemize}

We have no explanations for any of these phenomenon, but given the nature of how the matrix is formed believe this is more than coincidence.  We look forward to more research being done into these cards and matrices.

\end{document}